\documentclass[10pt]{article}
\usepackage{bbm}
\usepackage{amsmath}
\usepackage{amsthm}
\usepackage{tabularx}
\usepackage{enumitem}
\usepackage{amssymb} 
\usepackage{mathrsfs}
\usepackage{caption}
\usepackage{graphicx}
\usepackage{authblk}
\usepackage[hypertexnames=false,colorlinks=true,urlcolor=blue,linkcolor=blue,citecolor=blue]{hyperref}
\usepackage[numbers,comma,square,sort&compress]{natbib}
\usepackage[letterpaper,text={6.5in,9in},centering]{geometry}

\makeatletter\@addtoreset{equation}{section}\makeatother

\newtheorem{thm}{Theorem}[section]

\newtheorem{prop}[thm]{Proposition}

\newtheorem{rmk}[thm]{Remark}

 {\begin{trivlist}\item[]\textbf{Acknowledgments.}}{\end{trivlist}}
\usepackage{tikz}

\graphicspath{{Figures/}}

\begin{document}

\title{Radial amplitude equations for fully localised planar patterns}
\author[1]{Dan J. Hill}
\author[2]{David J. B. Lloyd}

\affil[1]{\small Fachrichtung Mathematik, Universit\"at des Saarlandes, Postfach 151150, 66041 Saarbr\"ucken, Germany}
\affil[2]{\small Department of Mathematics, University of Surrey, Guildford, GU2 7XH, UK}

\date{}
\maketitle

%%%%%%%%%%%%%%%%%%%%%%%%%%%%%%%%%%%%%%%%%
\begin{abstract}
Isolated patches of spatially oscillating pattern have been found to emerge near a pattern-forming instability in a wide variety of experiments and mathematical models. However, there is currently no mathematical theory to explain this emergence or characterise the structure of these patches.
We provide a method for formally deriving radial amplitude equations to planar patterns via non-autonomous multiple-scale analysis and convolutional sums of products of Bessel functions. Our novel approach introduces nonautonomous differential operators, which allow for the systematic manipulation of Bessel functions, as well as previously unseen identities involving infinite sums of Bessel functions. Solutions of the amplitude equations describe fully localised patterns with non-trivial angular dependence, where localisation occurs in a purely radial direction. Amplitude equations are derived for multiple examples of patterns with dihedral symmetry, including fully localised hexagons and quasipatterns with twelve-fold rotational symmetry. In particular, we show how to apply the asymptotic method to the Swift--Hohenberg equation and general reaction-diffusion systems. 
%%%%%%%%%%%%%%%%%%%%%%%%%%%%%%%%%%%%%%%%%%%%%%%%%%%%%%%%%%%%%%%%%%%%%%%%%%%%%%%%%%%%%%%%%%%%

\end{abstract}
%
%
%
%
%
%%%%%%%%%%%%%%%%%
%   Section 1   %
%%%%%%%%%%%%%%%%%
%
%
%
%
%
\section{Introduction}\label{s:intro}
The emergence of fully localised patterns---consisting of compact patches of pattern surrounded by a uniform state---have been observed in a wide variety of physical and mathematical models~\cite{Lloyd2008LocalizedHexagons,Subramanian2018LocalisedQC,knobloch2008spatially,knobloch2015spatial}. Numerically, these patches are found to emerge from a pattern-forming or Turing instability (see Figure~\ref{fig:hex_bif}) but there currently does not exist a mathematical theory to explain their formation from quiescence. The prototypical equation that exhibits these localised structures is the planar Swift--Hohenberg equation (SHE) given by~\cite{SwiftHohenberg1977Hydrodynamic}
\begin{equation}\label{e:SHE}
    \partial_t u = -(1+\Delta)^2 u - \mu u + \nu u^2 - u^3
\end{equation}
where $u=u(t,x,y)\in\mathbb{R}$, $\Delta:= \partial_{x}^{2} + \partial_{y}^{2}$ is the planar Laplacian, $\nu\in\mathbb{R}$ is fixed, and $0<\mu\ll1$ acts as a bifurcation parameter such that there is a pattern-forming instability at $\mu=0$. Crucially, it is observed that near onset the patches appear to have a solution structure described by the product of a slowly varying radial envelope with a domain-covering pattern. 
\begin{figure}[t]
    \centering
    \includegraphics[width=\linewidth]{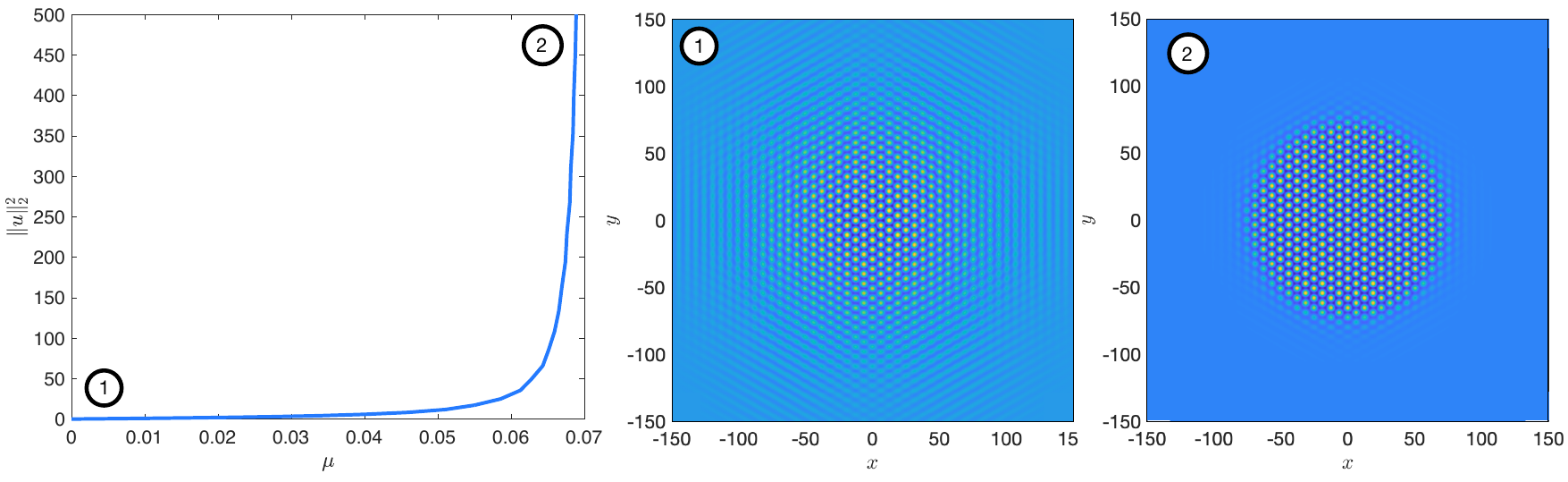}
    \caption{Bifurcation diagram of a localised hexagon patch in the Swift--Hohenberg equation~(\ref{e:SHE}) for $\nu=0.9$. The computations are carried out as described in~\cite{Lloyd2008LocalizedHexagons}.}
    \label{fig:hex_bif}
\end{figure}

In one spatial dimension, there is a well-established weakly nonlinear approach to studying the emergence of localised patterns near $\mu=0$, via multiple-scales analysis. Here, one formally derives an equation for a slowly varying envelope over a fast varying periodic pattern by carrying out an expansion of the form
\begin{equation*}
u(t,x) = \varepsilon\left(A(T,X)e^{ix} + c.c. \right) + \mathcal{O}(\varepsilon^2),\qquad \mu = \varepsilon^2\hat\mu,\qquad (T,X) = (\varepsilon^2t, \varepsilon x),
\end{equation*}
where $0<\varepsilon\ll1$ is the asymptotic small parameter, $A\in\mathbb{C}$ and $c.c$ stands for the complex conjugate of the preceding term. Carrying out a multiple-scales analysis (e.g.~\cite{Burke2007Normal}) for the SHE leads to an amplitude equation for $A$ given by the cubic Ginzburg--Landau equation
\begin{equation}\label{e:Ginz-1D}
    A_T = 4A_{XX} - \hat{\mu} A + 4\left(\frac{19\nu^2}{18} - \frac{3}{4}\right) |A|^2A.
\end{equation}
Justification of this approach in one dimension has been carried out in~\cite{Melbourne2004ComplexGL,Melbourne2004RealGL,Schneider1996GeneralizedGL}. In particular, it was proven by Schneider \cite{Schneider1994ErrorEstimatesGL} that the Ginzburg--Landau approximation results in an error of size $\mathcal{O}(\varepsilon^{\frac{3}{2}})$ for all $t<T_{\varepsilon} = \mathcal{O}(\varepsilon^{-2})$.

The one dimensional analysis was later extended to the radial (axisymmetric) case in~\cite{scheel2003radially,Lloyd2009LocalizedRadial,McCalla2013Spots,Hill2022LocalizedVegetation} where they considered stationary axisymmetric solutions of (\ref{e:SHE}) and carried out rigorous radial normal form analysis. Their proofs consist of matching a core-manifold of solutions (i.e. solutions that remain bounded as $r\rightarrow0$) with a far-field manifold (i.e. solutions that decay to zero exponentially fast in $r$) to construct a variety of localised radial solutions. Localised ring solutions in the far-field manifold were then found to solve a nonautonomous variant of the Ginzburg-Landau equation, given by
\begin{equation*}
0 = 4\left(\frac{\mathrm{d}}{\mathrm{d}R} + \frac{1}{2R}\right)^2 A - \hat{\mu} A + 4\left(\frac{19\nu^2}{18} - \frac{3}{4}\right) |A|^2A,
\end{equation*}
where $R=\varepsilon r$ is a slow spatial scale. 

Recently, Hill et al.~\cite{Hill2023DihedralSpots,Hill2024DihedralRings} developed an approximate theory for the emergence of fully localised planar patterns with dihedral symmetry. In particular, they studied the existence of patches and rings for general two-component reaction-diffusion systems near a pattern-forming instability by carrying out a finite-Fourier decomposition in the angular coordinate. This decomposition resulted in a large system of radial ODEs, for which they performed a similar rigorous radial analysis as in the axisymmetric case and showed the existence of various approximate localised patches and rings. 

Asymptotic analysis of general 2D localised patterns is significantly more difficult, as one can not simply reduce to an ordinary differential equation. Amplitude equations for planar hexagons have been derived where one carries out an asymptotic expansion of the form
\begin{equation*}
u(t,x,y) = \varepsilon\,\left(A_1(T,X,Y)e^{ix} + A_2(T,X,Y)e^{i(x+\sqrt{3}y)/2} + A_3(T,X,Y)e^{i(x-\sqrt{3}y)/2} + c.c.\right) + \mathcal{O}(\varepsilon^2),
\end{equation*}
where $(T,X,Y)=(\varepsilon^2 t, \varepsilon x,\varepsilon y)$ and $A_i$ are complex functions; see for instance~\cite{Hoyle2006} for a review. One can then derive three coupled amplitude equations for each of the $A_i$'s. However, even in the case where each $A_i$ only depends on one slow spatial variable, these equations have no explicit solution and their analysis remains a significant challenge; see for example~\cite{Doelman2003hexagon}.  

In this paper, we instead extend the weakly nonlinear analysis of~\cite{Burke2007Normal} to polar coordinates. In particular, we formally look for a slowly varying axisymmetric envelope over a domain-covering pattern, using asymptotic techniques to derive a radial amplitude equation. Using this approach, we formally show the bifurcation of 4 different types of fully localised patterns; stripes, hexagons, rhomboids, and a twelve-fold quasipattern, as shown in Figure~\ref{fig:loc-profiles}. We expand $u$ in an angular Fourier series given by
\begin{equation*}
    u(t,x,y) = \tilde{u}(t,r,\theta) = \sum_{n\in\mathbb{Z}} u_{n}(t,r)\mathrm{e}^{\mathrm{i}n\theta},
\end{equation*}
with the reality condition $u_{-n} = \overline{u_n}$ for all $n\in\mathbb{Z}$, such that the SHE \eqref{e:SHE} becomes
\begin{equation}\label{e:SHE-n}
    \partial_t u_n = -(1+\Delta_n)^2 u_n - \mu u_n + \nu \sum_{i+j=n}u_{i}u_{j} - \sum_{i+j+k=n}u_{i}u_{j}u_{k}
\end{equation} 
for each $n\in\mathbb{Z}$, where $\Delta_{n}:=\partial_{r}^{2} + \frac{1}{r}\partial_{r} - \frac{n^2}{r^2}$. We note that the angular Fourier projection has transformed each nonlinear term in \eqref{e:SHE} into a convolutional sum in \eqref{e:SHE-n}, which we exploit below. We then look for solutions near the pattern-forming instability by carrying out an expansion of the form
\begin{equation*}
u_n(t,r) = \varepsilon\,\left( A(T,R)\, a_n J_{n}(r) + \overline{A(T,R)}\, a_n J_{-n}(r)\right) + \mathcal{O}(\varepsilon^2), \qquad \mu = \varepsilon^2\hat{\mu},\qquad (T,R) = (\varepsilon^2 t, \varepsilon r),
\end{equation*}
where $A\in\mathbb{C}$ is a slowly varying envelope, $\{a_k\}_{k\in\mathbb{Z}}$ are fixed coefficients that define the geometry of the domain-covering pattern, and $J_{n}(r)$ is the $n$-th order Bessel function of the first kind. The key here is that the slowly varying envelope is the same for each Fourier mode $u_n$, allowing us to construct a single amplitude equation for $A$. The amplitude $A$ can then be thought of as an axisymmetric amplitude of the planar function $u(t,x,y)$.

We note that the use of Bessel functions $J_n(r)$ for $n\in\mathbb{Z}$ here is not a particular choice, but rather is enforced by the form of the linear operator $(1+\Delta_n)^2$ of \eqref{e:SHE-n}. In weakly nonlinear analysis, one wants to decompose the solution onto the eigenfunctions of the time-independent linear operator, so that one eigenfunction lies in the kernel of the operator and the remaining elements lie in the range of the operator. This is so that higher order correction terms can be used to remove nonlinear resonances of lower order terms. For example, in one spatial dimension the complex exponential functions $\mathrm{e}^{\mathrm{i}\lambda x}$ act as eigenfunctions of the linear operator $(1+\partial_{x}^2)^2$ and, in particular $(1+\partial_{x}^2)^2\mathrm{e}^{\mathrm{i}x} = 0$. Hence, in the one dimensional weakly nonlinear analysis, leading order solutions depend on $\mathrm{e}^{\mathrm{i}x}$ and higher order corrections are decomposed onto the eigenfunctions $\mathrm{e}^{\mathrm{i}\lambda x}$. Returning to the radial linear operator in \eqref{e:SHE-n}, we note that the equation
\begin{equation}\label{e:Bessel}
    (1 + \Delta_n)u_n = \left(\partial_{r}^{2} + \frac{1}{r}\partial_{r} + \left(1 -  \frac{n^2}{r^2}\right)\right)u_n = 0
\end{equation}
is Bessel's equation, and $J_n(\lambda r)$ is an eigenfunction of $(1+\Delta_n)^2$ with eigenvalue $(1-\lambda^2)^2$. Hence, the eigenfunction $J_n(r)$ lies in the kernel of $(1+\Delta_n)^2$ for each $n\in\mathbb{Z}$, and so we require our leading order solution to depend on $J_{n}(r)$ and higher order corrections to be decomposed onto the eigenfunctions $J_{n}(\lambda r)$.

\begin{figure}[t]
    \centering
    \includegraphics[width=\linewidth]{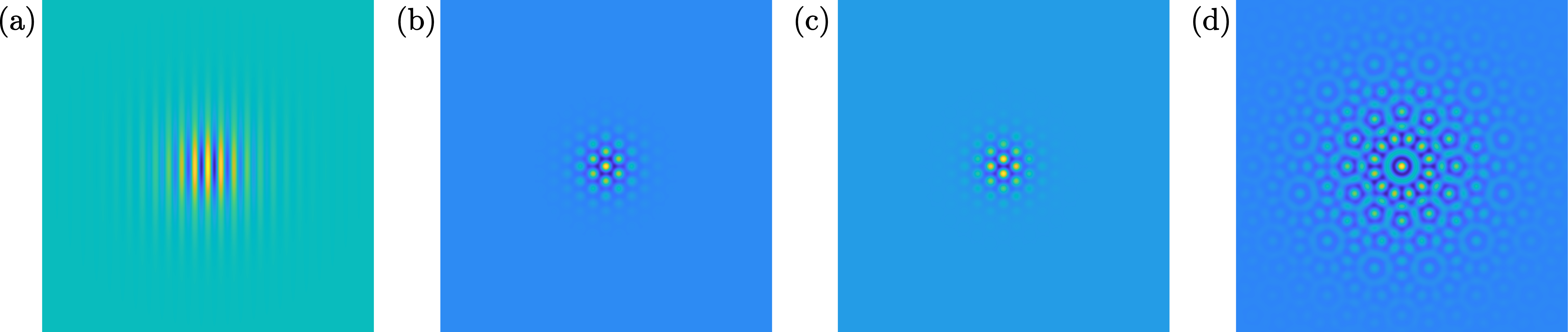}
    \caption{Examples of fully localised (a) stripes, (b) on-centre hexagons, (c) off-centre hexagons, and (d) a twelve-fold quasipattern}
    \label{fig:loc-profiles}
\end{figure}

In order to carry out the multiple-scales analysis in polar coordinates, several problems need to be overcome. The first is that Bessel functions do not work well with the standard radial differential operator $\partial_r$. For example, 
\begin{equation*}
\partial_r J_{n}(r) = \frac{1}{2}\left(J_{n-1}(r) - J_{n+1}(r)\right) = J_{n-1}(r) - \frac{n}{r}J_{n}(r) = -J_{n+1}(r) + \frac{n}{r}J_{n}(r).
\end{equation*}
for an arbitrary $n\in\mathbb{Z}$, and so applying $\partial_r$ to $J_{n}(r)$ either requires keeping track of two sets of shifting indices or introduces additional nonautonomous terms into the problem. Hence, we need to find a method to simplify the differentiation of Bessel functions. We do this by introducing what we call {\it Bessel differential operators}
\begin{equation*}
\mathcal{D}_{n}:= \partial_r + \frac{n}{r},\qquad n\in\mathbb{Z}.
\end{equation*}
These operators form an integral part of the recent radial function space theory of \cite{GrovesHill2024FunctionSpaces}, where $\mathcal{D}_{n}, \mathcal{D}_{-n}$ act as the natural differential operators for an $n$-th angular Fourier mode $u_n(r)$. For our purposes, the operators $\mathcal{D}_{n}, \mathcal{D}_{-n}$ have the convenient property that 
\begin{equation*}
\mathcal{D}_{n} J_{n}(r) = J_{n-1}(r),\qquad \qquad \mathcal{D}_{-n} J_{n}(r) = -J_{n+1}(r),
\end{equation*}
for all $n\in\mathbb{Z}$, and so they effectively act as shift operators on the index $n$ of $J_n(r)$. By using these Bessel operators, we are able to preserve much of the radial harmonic structure of the operator $\Delta_n$ in our multiple-scales analysis.

The second problem is how to analyse nonlinear terms involving Bessel functions. In particular, we need to be able to compare nonlinear convolutional sums of Bessel functions with higher-frequency terms $J_{k}(m r)$. We derive several useful identities, including
\begin{equation*}
    \sum_{i+j+k=n} J_{i}(a r)J_{j}(b r)J_{\pm k}(c r) = J_{n}((a+b\pm c)r ),
\end{equation*}
which solves this problem. In this case, we can thus treat products of Bessel functions as radial analogues of the complex exponential function; e.g.,
    \begin{equation*}
        \mathrm{e}^{\mathrm{i}r}\mathrm{e}^{\mathrm{i}4r}\mathrm{e}^{-\mathrm{i}3r} = \mathrm{e}^{\mathrm{i}2r}, \qquad\quad\sim\qquad \sum_{i+j+k=n} J_{i}(r)J_{j}(4r)J_{-k}(3r) = J_{n}(2r).
    \end{equation*}
When replacing each Bessel function $J_{i}(r)$ by $a_i J_i(r)$, the resulting identities become significantly more complicated. However, the method to derive these identities remains simple and is very easy to automate. In this work, we restrict to a small number of particular cases whose identities we derive in \S\ref{ss:bessel-nonlin}. To the authors' knowledge, the identities presented in \S\ref{ss:bessel-nonlin} have not been stated elsewhere, other than the folkloric result that

\begin{equation*}
    \sum_{i\in\mathbb{Z}} J_{i}(r)J_{i}(r) = 1.
\end{equation*}

However, the derivation of each identity in \S\ref{ss:bessel-nonlin} is so straightforward that it would be na\"ive to assume it has not been proven up until now. One aim of this work is then to promote the knowledge of these identities to the wider community and highlight the practicality of using Bessel functions in the study of fully localised planar patterns.

The rest of the paper is outlined as follows. In \S\ref{s:toolbox} we show how to expand the linear SHE operator $(1+\Delta_n)^2$ in terms of the Bessel differential operators $\mathcal{D}_n$ with multiple scales $r$ and $R=\varepsilon r$, before then deriving various important identities regarding convolutional sums of products of Bessel functions. Then, in \S\ref{s:amplitude} we carry out the multiple scales asymptotic expansion for the SHE to derive radial amplitude equations for fully localised stripes, hexagons, rhomboids and a twelve-fold quasipattern. We extend this approach to a general reaction--diffusion system in \S\ref{s:RD}, where we again derive an amplitude equation for fully localised hexagons. Finally, we conclude in \S\ref{s:Discussion} with a discussion of the strengths and limitations of our approach, as well as the next possible extensions of this work.

%
%
%
%
%
%%%%%%%%%%%%%%%%%
%   Section 2   %
%%%%%%%%%%%%%%%%%
%
%
%
%
%
\section{Bessel functions: differentiation and nonlinearity}\label{s:toolbox}
The key issues when considering radial differential equations such as \eqref{e:SHE-n} are the presence of non-autonomous differential operators and maintaining control of the nonlinear terms. To alleviate each of these problems, we present the key tools that will allow us to derive radial amplitude equations in \S\ref{s:amplitude}. We begin by introducing nonautonomous differential operators that will be convenient when working with Bessel functions. Following this, we derive identities involving convolutional sums of Bessel functions which will enable us to simplify nonlinear terms in our differential equation.

%%%%%%%%%%%%%%%%%%%%%%%%%%%%%%%%%%%%%%%%%%%%
\subsection{Radial differential operators}\label{ss:bessel-deriv}
We first explore the role of differentiation in the analysis of Bessel functions. To simplify the differentiation of Bessel functions, we introduce $n$-indexed nonautonomous differential operators $\mathcal{D}_{\pm n} := \partial_r \pm \frac{n}{r}$ for any $n\in\mathbb{Z}$, which we call \emph{Bessel differential operators}. These operators have the convenient property that
\begin{equation*}
\mathcal{D}_{n} J_{n}(r) = J_{n-1}(r),\qquad \qquad \mathcal{D}_{-n} J_{n}(r) = -J_{n+1}(r),
\end{equation*}
for all $n\in\mathbb{Z}$, and so they effectively act as shift operators on the index $n$ of $J_n(r)$. The Bessel operators satisfy the following commutation relation $\mathcal{D}_{m}\mathcal{D}_{n} = \mathcal{D}_{n+1}\mathcal{D}_{m-1}$ for any $n,m\in\mathbb{Z}$, and we note that 
\begin{equation*}
\mathcal{D}_{1-n}\mathcal{D}_{n} = \partial_r^2 + \frac{1}{r}\partial_r - \frac{n^2}{r^2} = \Delta_n.    
\end{equation*}
where Bessel's equation can then be written as $(1+\mathcal{D}_{1-n}\mathcal{D}_{n})u=0$. 

In order to follow the approach of Burke and Knobloch \cite{Burke2007Normal} to formally derive radial amplitude equations for \eqref{e:SHE-n}, we must first introduce radial multiple-scale analysis which we design to be convenient when working with Bessel functions. We introduce a slow radial variable $R:=\varepsilon r$, and define the $n$-indexed slow differential operators $\hat{\mathcal{D}}_{\pm n}:=\left(\partial_{R} \pm \frac{n}{R}\right)$ and $\hat{\Delta}_{n}:=\hat{\mathcal{D}}_{1-n}\hat{\mathcal{D}}_{n}$. We consider a function $u_n = u_n(r,R)$, where we assume that $u_n$ acts like an $J_n$ Bessel function with respect to $r$ and like an axisymmetric function with respect to $R$. This assumption means we choose to apply the operators $\{\mathcal{D}_{\pm n}, \Delta_{n}\}$ and $\{\hat{\mathcal{D}}_{0},\hat{\Delta}_{0}\}$ to $u_{n}$, respectively.

We note that, for any function $f(r,R)$, we can write
\begin{equation*}
    \left(\tfrac{\mathrm{d}}{\mathrm{d}r} + \tfrac{n}{r}\right)f = \left(\partial_{r} + \tfrac{k}{r}\right)f + \varepsilon\,\left(\partial_{R} + \tfrac{n-k}{R}\right)f = \mathcal{D}_{k}f + \varepsilon\,\hat{\mathcal{D}}_{n-k}f
\end{equation*}
for any $k\in\mathbb{Z}$. Hence, when considering the function $u_n$ we choose the index $k$ so that we are always applying a convenient Bessel operator to the Bessel function that we assume $u_n$ acts like. In particular, we see that
\begin{equation*}
    \begin{split}
        \Delta_{n}u_n ={}& \left(\tfrac{\mathrm{d}}{\mathrm{d}r} + \tfrac{1-n}{r}\right)\left[\mathcal{D}_{n}u_n + \varepsilon\,\hat{\mathcal{D}}_{0}u_n\right]\\
        ={}& \left(\mathcal{D}_{1-n} + \varepsilon\,\hat{\mathcal{D}}_{0}\right)\mathcal{D}_{n}u_n + \varepsilon\,\left(\mathcal{D}_{-n} + \varepsilon\,\hat{\mathcal{D}}_{1}\right)\hat{\mathcal{D}}_{0}u_n\\
        ={}& \mathcal{D}_{1-n}\mathcal{D}_{n}u_n + \varepsilon\left( \mathcal{D}_{n} + \mathcal{D}_{-n}\right)\hat{\mathcal{D}}_{0}u_n + \varepsilon^2\,\hat{\mathcal{D}}_{1}\hat{\mathcal{D}}_{0}u_n.\\
    \end{split}
\end{equation*}
Note that for the term $\hat{\mathcal{D}}_0 u_n$ we have chosen $(\tfrac{\mathrm{d}}{\mathrm{d}r} + \tfrac{1-n}{r}) = (\mathcal{D}_{-n} + \varepsilon\,\hat{\mathcal{D}}_{1})$, since we expect $\hat{\mathcal{D}}_0 u_n$ to act like $J_n(r)$, whereas for the term $\mathcal{D}_n u_n$ we have chosen $(\tfrac{\mathrm{d}}{\mathrm{d}r} + \tfrac{1-n}{r}) = (\mathcal{D}_{1-n} + \varepsilon\,\hat{\mathcal{D}}_{0})$, since we expect $\mathcal{D}_{n}u_n$ to act like $\mathcal{D}_{n}J_n(r) = J_{n-1}(r)$. Thus, we have so far determined the following
\begin{equation*}
    \begin{split}
        \left(1 + \Delta_{n}\right)u_n ={}& \left(1 + \mathcal{D}_{1-n}\mathcal{D}_{n}\right)u_n + \varepsilon\left( \mathcal{D}_{n} + \mathcal{D}_{-n}\right)\hat{\mathcal{D}}_{0}u_n + \varepsilon^2\,\hat{\mathcal{D}}_{1}\hat{\mathcal{D}}_{0}u_n.\\
    \end{split}
\end{equation*}
Applying the operator $(\tfrac{\mathrm{d}}{\mathrm{d}r} + \tfrac{n}{r})$, we obtain
\begin{equation*}
    \begin{split}
        (\tfrac{\mathrm{d}}{\mathrm{d}r} + \tfrac{n}{r})\left(1 + \Delta_{n}\right)u_n ={}& (\tfrac{\mathrm{d}}{\mathrm{d}r} + \tfrac{n}{r})\left[\left(1 + \mathcal{D}_{1-n}\mathcal{D}_{n}\right)u_n + \varepsilon \mathcal{D}_{n}\hat{\mathcal{D}}_{0}u_n + \varepsilon \mathcal{D}_{-n}\hat{\mathcal{D}}_{0}u_n + \varepsilon^2\,\hat{\mathcal{D}}_{1}\hat{\mathcal{D}}_{0}u_n\right],\\
        ={}& \left(\mathcal{D}_{n} + \varepsilon\,\hat{\mathcal{D}}_{0}\right)\left(1 + \mathcal{D}_{1-n}\mathcal{D}_{n}\right)u_n + \varepsilon \left(\mathcal{D}_{n-1} + \varepsilon\,\hat{\mathcal{D}}_{1}\right)\mathcal{D}_{n}\hat{\mathcal{D}}_{0}u_n \\
        &\qquad + \varepsilon \left(\mathcal{D}_{n+1} + \varepsilon\,\hat{\mathcal{D}}_{-1}\right)\mathcal{D}_{-n}\hat{\mathcal{D}}_{0}u_n + \varepsilon^2\,\left(\mathcal{D}_{n} + \varepsilon\,\hat{\mathcal{D}}_{0}\right)\hat{\mathcal{D}}_{1}\hat{\mathcal{D}}_{0}u_n,\\
        ={}& \mathcal{D}_{n}\left(1 + \mathcal{D}_{1-n}\mathcal{D}_{n}\right)u_n + \varepsilon\,\left[\left(1 + \mathcal{D}_{1-n}\mathcal{D}_{n}\right)\hat{\mathcal{D}}_{0}u_n +  \left(\mathcal{D}_{n-1} + \mathcal{D}_{1-n}\right)\mathcal{D}_{n}\hat{\mathcal{D}}_{0}u_n\right]\\
        &\qquad + \varepsilon^2 \left[2\mathcal{D}_{n}\hat{\mathcal{D}}_{1}\hat{\mathcal{D}}_{0}u_n + \mathcal{D}_{-n}\hat{\mathcal{D}}_{-1}\hat{\mathcal{D}}_{0}u_n\right] + \mathcal{O}(\varepsilon^3).\\
    \end{split}
\end{equation*}
Finally, we apply $(\tfrac{\mathrm{d}}{\mathrm{d}r} + \frac{1-n}{r})$, resulting in 
\begin{equation*}
    \begin{split}
        \Delta_{n}\left(1 + \Delta_{n}\right)u_n ={}& \left(\mathcal{D}_{1-n} + \varepsilon \hat{\mathcal{D}}_{0}\right)\mathcal{D}_{n}\left(1 + \mathcal{D}_{1-n}\mathcal{D}_{n}\right)u_n+ \varepsilon\,\left[\left(\mathcal{D}_{-n} + \varepsilon \hat{\mathcal{D}}_{1}\right)\left(1 + \mathcal{D}_{1-n}\mathcal{D}_{n}\right)\hat{\mathcal{D}}_{0}u_n\right]\\
        & + \varepsilon\,\left[\left(\mathcal{D}_{2-n} + \varepsilon \hat{\mathcal{D}}_{-1}\right)\mathcal{D}_{n-1}\mathcal{D}_{n}\hat{\mathcal{D}}_{0}u_n + \left(\mathcal{D}_{-n} + \varepsilon \hat{\mathcal{D}}_{1}\right)\mathcal{D}_{1-n}\mathcal{D}_{n}\hat{\mathcal{D}}_{0}u_n\right]\\
        & + \varepsilon^2 \left[2\left(\mathcal{D}_{1-n} + \varepsilon \hat{\mathcal{D}}_{0}\right)\mathcal{D}_{n}\hat{\mathcal{D}}_{1}\hat{\mathcal{D}}_{0}u_n + \left(\mathcal{D}_{-(n+1)} + \varepsilon \hat{\mathcal{D}}_{2}\right)\mathcal{D}_{-n}\hat{\mathcal{D}}_{-1}\hat{\mathcal{D}}_{0}u_n\right] + \mathcal{O}(\varepsilon^3),\\
        ={}& \mathcal{D}_{1-n}\mathcal{D}_{n}\left(1 + \mathcal{D}_{1-n}\mathcal{D}_{n}\right)u_n + \varepsilon\,\left[\left(\mathcal{D}_{n} + \mathcal{D}_{-n}\right)\left(1 + 2\mathcal{D}_{1-n}\mathcal{D}_{n}\right)\hat{\mathcal{D}}_{0}u_n\right]\\
        & + \varepsilon^2 \left[\left(1 + 4\mathcal{D}_{1-n}\mathcal{D}_{n}\right)\hat{\mathcal{D}}_{1}\hat{\mathcal{D}}_{0}u_n + \mathcal{D}_{-1-n}\mathcal{D}_{-n}\hat{\mathcal{D}}_{-1}\hat{\mathcal{D}}_{0}u_n + \mathcal{D}_{n-1}\mathcal{D}_{n}\hat{\mathcal{D}}_{-1}\hat{\mathcal{D}}_{0}u_n\right] + \mathcal{O}(\varepsilon^3),\\
    \end{split}
\end{equation*}
which implies that
\begin{equation*}
    \begin{split}
        \left(1 + \Delta_{n}\right)^2 u_n ={}& \left(1 + \mathcal{D}_{1-n}\mathcal{D}_{n}\right)^2 u_n + \varepsilon\,\left[2\left(\mathcal{D}_{n} + \mathcal{D}_{-n}\right)\left(1 + \mathcal{D}_{1-n}\mathcal{D}_{n}\right)\hat{\mathcal{D}}_{0}u_n\right]\\
        & + \varepsilon^2 \left[2\left(1 + 2\mathcal{D}_{1-n}\mathcal{D}_{n}\right)\hat{\mathcal{D}}_{1}\hat{\mathcal{D}}_{0}u_n + \mathcal{D}_{-1-n}\mathcal{D}_{-n}\hat{\mathcal{D}}_{-1}\hat{\mathcal{D}}_{0}u_n + \mathcal{D}_{n-1}\mathcal{D}_{n}\hat{\mathcal{D}}_{-1}\hat{\mathcal{D}}_{0}u_n\right] + \mathcal{O}(\varepsilon^3).\\
    \end{split}
\end{equation*}
For simplicity, we would like to express as many differential operators for $r$ ($R$) in terms of the Laplacian operator $\Delta_n = \mathcal{D}_{1-n}\mathcal{D}_{n}$ ($\hat{\Delta}_n = \hat{\mathcal{D}}_{1-n}\hat{\mathcal{D}}_{n}$) or just $\partial_{r} = \mathcal{D}_{0}$ ($\partial_{R} = \hat{\mathcal{D}}_{0}$). To this end, we write
\begin{equation*}
    \begin{split}
        \left(1 + \Delta_{n}\right)^2 u_n ={}& \left(1 + \Delta_{n}\right)^2 u_n + \varepsilon\,\left[4\partial_{r}\left(1 + \Delta_{n}\right)\partial_{R} u_n\right]\\
        & + \varepsilon^2 \left[2\left(1 + \Delta_{n}\right)\hat{\Delta}_{0}u_n + 2\Delta_{n}\hat{\mathcal{D}}_{1}\hat{\mathcal{D}}_{0}u_n + \mathcal{D}_{-1-n}\mathcal{D}_{-n}\hat{\mathcal{D}}_{-1}\hat{\mathcal{D}}_{0}u_n + \mathcal{D}_{n-1}\mathcal{D}_{n}\hat{\mathcal{D}}_{-1}\hat{\mathcal{D}}_{0}u_n\right] + \mathcal{O}(\varepsilon^3),\\
        ={}& \left(1 + \Delta_{n}\right)^2 u_n + \varepsilon\,\left[4\partial_{r}\left(1 + \Delta_{n}\right)\partial_{R} u_n\right] + \varepsilon^2 \left[2\left(1 + \Delta_{n}\right)\hat{\Delta}_{0}u_n + 4\Delta_{n}\partial_{R}^{2} u_n\right]\\
        & + \varepsilon^2 \left[\left(\mathcal{D}_{n-1} - \mathcal{D}_{1-n}\right)\mathcal{D}_{n} - \left(\mathcal{D}_{n+1} - \mathcal{D}_{-1-n}\right)\mathcal{D}_{-n}\right]\hat{\mathcal{D}}_{-1}\hat{\mathcal{D}}_{0}u_n + \mathcal{O}(\varepsilon^3),\\
        ={}& \left(1 + \Delta_{n}\right)^2 u_n + \varepsilon\,\left[4\partial_{r}\left(1 + \Delta_{n}\right)\partial_{R} u_n\right] + \varepsilon^2 \left[2\left(1 + \Delta_{n}\right)\hat{\Delta}_{0}u_n + 4\Delta_{n}\partial_{R}^{2} u_n\right] + \mathcal{O}(\varepsilon^3).\\
    \end{split}
\end{equation*}
where we have used the fact that $\mathcal{D}_{n} + \mathcal{D}_{-n} = 2\,\partial_r$ and $\mathcal{D}_{n} - \mathcal{D}_{-n} = \frac{2n}{r} = \varepsilon\, \frac{2n}{R}$ for all $n\in\mathbb{Z}$. Hence, we have arrived at the convenient asymptotic expansion
\begin{equation}\label{diff:asymp}
    \begin{split}
        \left(1 + \Delta_{n}\right)^2 u_n ={}& \left(1 + \Delta_{n}\right)^2 u_n + \varepsilon\,\left[4\partial_{r}\left(1 + \Delta_{n}\right)\partial_{R} u_n\right] + \varepsilon^2 \left[2\left(1 + \Delta_{n}\right)\hat{\Delta}_{0}u_n + 4\Delta_{n}\partial_{R}^{2} u_n\right] + \mathcal{O}(\varepsilon^3).\\
    \end{split}
\end{equation}

\begin{rmk}
    It is strange to see nonautonomous terms in a multiple-scale asymptotic expansion, and so we emphasise why we choose certain nonautonomous terms to be in terms of the ``fast'' variable $r$ rather than the ``slow'' variable $R$. 

Suppose we chose to write every nonautonomous term in terms of the slow variable $R$, so that the leading order differential operator becomes $(1+\partial_r^2)^2$. Applying the operator $(1+\partial_{r}^2)$ to $J_n(r)$, we note that
\begin{equation*}
    (1+\partial_r^2) J_{n}(r) = - \frac{1}{r}\partial_r J_{n}(r) + \frac{n^2}{r^2}J_{n}(r) = - \varepsilon\,\frac{1}{R}\partial_r J_{n}(r) + \varepsilon^2\,\frac{n^2}{R^2}J_{n}(r) 
\end{equation*}
and so we obtain additional higher order nonautonomous terms in our asymptotic expansion. These higher order terms will then combine with the other nonautonomous terms in our asymptotic expansion, which will require the application of various Bessel function identities in order to simplify. Hence, our choice of preserving certain ``fast'' nonautonomous terms is purely for convenience, as it allows us to present a simpler method for radial multiple-scale analysis in \S\ref{s:amplitude}.
\end{rmk}

%%%%%%%%%%%%%%%%%%%%%%%%%%%%%%%%%%%%%%%%%%%%
\subsection{Convolutional sums of Bessel functions}\label{ss:bessel-nonlin}

Having resolved the issue surrounding derivatives of Bessel functions, we now turn our attention towards the nonlinear terms in \eqref{e:SHE-n}. We present identities for the summation of convolutional products of Bessel functions, such that the indices $k_i\in\mathbb{Z}$ of each Bessel function sum up to a fixed index $n\in\mathbb{Z}$. As we already noted when deriving \eqref{e:SHE-n}, the projection of polynomial nonlinearities onto a given angular Fourier mode $\mathrm{e}^{\mathrm{i}n\theta}$ results in summations of exactly this form.\par

We begin with the following general result for combinations of Bessel functions,
\begin{prop}\label{prop:Besselsums}
    Fix $n\in\mathbb{Z}$, $m\in\mathbb{N}$, and $\mathbf{k}:=(k_1,k_2,\dots,k_{m})\in\mathbb{Z}^{m}$, with $|\mathbf{k}|:=k_1 + k_2 + \dots + k_m$. Then,
    \begin{equation*}
        \sum_{|\mathbf{k}|=n}\,\prod_{i=1}^{m} \mathrm{e}^{\mathrm{i} k_i\,\mathrm{arg}(\mathbf{x}_{i})}\,J_{k_{i}}(|\mathbf{x}_{i}|) = J_{n}\left(\left|\sum_{i=1}^{m} \mathbf{x}_{i}\right|\right)\,\mathrm{e}^{\mathrm{i}n\,\mathrm{arg}\left(\sum_{i=1}^{m} \mathbf{x}_{i}\right)}
    \end{equation*}
    holds for all $\mathbf{x}_{1},\mathbf{x}_{2},\dots,\mathbf{x}_{m}\in \mathbb{R}^{2}$, where the convolutional sum is taken over all $\mathbf{k}\in\mathbb{Z}^{m}$, $|\mathbf{x}|$ denotes the modulus of $\mathbf{x}$, and $\arg(\mathbf{x})$ denotes the `argument' of $\mathbf{x}$, i.e. the anticlockwise angle from the positive $x$-axis to the vector $\mathbf{x}$.
\end{prop}
\begin{proof}
We begin with the Jacobi--Anger expansion \cite[(9.1.41)]{AbramowitzStegun1948Handbook}
\begin{equation}\label{ident:Jacobi-Anger}
    \mathrm{e}^{\mathrm{i}r\cos(\theta)} = \sum_{k\in\mathbb{Z}} J_{k}(r)\,\mathrm{e}^{\mathrm{i}k(\theta + \frac{\pi}{2})}
\end{equation}
and note that, by projecting onto the $n$-th angular Fourier mode, we recover the following result
\begin{equation}\label{ident:Hansen-Bessel}
    \frac{1}{2\pi}\int_0^{2\pi}\mathrm{e}^{\mathrm{i}(r\cos(\theta)-n(\theta+\frac{\pi}{2}))}\,\mathrm{d}\theta = J_{n}(r),
\end{equation}
which is known as the \emph{Hansen--Bessel formula}. For $\mathbf{x},\mathbf{y}\in\mathbb{R}^{2}$ it follows from \eqref{ident:Jacobi-Anger} that
\begin{equation}\label{ident:Jacobi-Anger;gen}
    \mathrm{e}^{\mathrm{i}\mathbf{y}\cdot\mathbf{x}} = \mathrm{e}^{\mathrm{i}|\mathbf{y}||\mathbf{x}|\cos(\mathrm{arg}(\mathbf{x}) - \mathrm{arg}(\mathbf{y}))} = \sum_{k\in\mathbb{Z}} J_{k}(|\mathbf{y}|\,|\mathbf{x}|)\,\mathrm{e}^{\mathrm{i} k\,(\mathrm{arg}(\mathbf{x}) - \mathrm{arg}(\mathbf{y}) + \frac{\pi}{2})}.
\end{equation}
Setting $\mathbf{y}(\varphi)=(\cos(\varphi+\frac{\pi}{2}),\sin(\varphi+\frac{\pi}{2}))^{T}$ for $\varphi\in[0,2\pi)$, we obtain 
\begin{equation*}
    \frac{1}{2\pi}\int_{0}^{2\pi}\mathrm{e}^{\mathrm{i}\left( n \varphi + \mathbf{y}(\varphi)\cdot\mathbf{x}\right)} \,\mathrm{d}\varphi = \sum_{k\in\mathbb{Z}} \left(\frac{1}{2\pi}\int_{0}^{2\pi}\mathrm{e}^{\mathrm{i}(n-k)\varphi} \,\mathrm{d}\varphi\right) J_{k}(|\mathbf{x}|)\mathrm{e}^{\mathrm{i} k\,\mathrm{arg}(\mathbf{x})} = J_{n}(|\mathbf{x}|)\,\mathrm{e}^{\mathrm{i} n\,\mathrm{arg}(\mathbf{x})}
\end{equation*}
and we can thus derive the \textit{generalised Hansen--Bessel formula}
\begin{equation}\label{ident:Hansen-Bessel;gen}
    \frac{1}{2\pi}\int_{0}^{2\pi}\mathrm{e}^{\mathrm{i} n \varphi }\,\mathrm{e}^{\mathrm{i}(\cos(\varphi+\frac{\pi}{2}),\sin(\varphi+\frac{\pi}{2}))^{T}\cdot\mathbf{x}} \,\mathrm{d}\varphi = J_{n}(|\mathbf{x}|)\,\mathrm{e}^{\mathrm{i} n\,\mathrm{arg}(\mathbf{x})}.
\end{equation}
Using \eqref{ident:Jacobi-Anger;gen} and \eqref{ident:Hansen-Bessel;gen}, we obtain
\begin{equation*}
    \begin{split}
J_{n}\left(\left|\sum_{i=1}^{m}\mathbf{x}_{i}\right|\right)\,\mathrm{e}^{\mathrm{i} n\,\mathrm{arg}\left(\sum_{i=1}^{m}\mathbf{x}_{i}\right)} ={}& \frac{1}{2\pi}\int_{0}^{2\pi}\mathrm{e}^{\mathrm{i} n \varphi } \mathrm{e}^{\mathrm{i}(\cos(\varphi+\frac{\pi}{2}),\sin(\varphi+\frac{\pi}{2}))^{T}\cdot\sum_{i=1}^{m}\mathbf{x}_{i}} \,\mathrm{d}\varphi\\
        ={}& \frac{1}{2\pi}\int_{0}^{2\pi}\mathrm{e}^{\mathrm{i}n \varphi}\, \prod_{i=1}^{m} \mathrm{e}^{\mathrm{i}(\cos(\varphi+\frac{\pi}{2}),\sin(\varphi+\frac{\pi}{2}))\cdot\mathbf{x}_{i}} \,\mathrm{d}\varphi\\
        ={}& \frac{1}{2\pi}\int_{0}^{2\pi}\mathrm{e}^{\mathrm{i}n \varphi}\, \prod_{i=1}^{m} \sum_{\mathbf{k}\in\mathbb{Z}^{m}} J_{k_i}(|\mathbf{x}|)\,\mathrm{e}^{\mathrm{i} k_i\,(\mathrm{arg}(\mathbf{x}_i) - \varphi)} \,\mathrm{d}\varphi\\
        ={}& \sum_{\mathbf{k}\in\mathbb{Z}^{m}} \left(\frac{1}{2\pi}\int_{0}^{2\pi}\mathrm{e}^{\mathrm{i}\left(n - \sum_{i=1}^{m}k_i\right)\varphi}\,\mathrm{d}\varphi\right) \, \prod_{i=1}^{m}  J_{k_i}(|\mathbf{x}|)\,\mathrm{e}^{\mathrm{i} k_i\,\mathrm{arg}(\mathbf{x}_i)} \\
        ={}& \sum_{|\mathbf{k}|=n} \,\prod_{i=1}^{m}  J_{k_i}(|\mathbf{x}|)\,\mathrm{e}^{\mathrm{i} k_i\,\mathrm{arg}(\mathbf{x}_i)}.\\
    \end{split}
\end{equation*}
\end{proof}
\begin{rmk}
    The convolutional identity presented in Proposition \ref{prop:Besselsums} is a generalisation of Graf's addition theorem (which can be found in Watson \cite[(2-3) in \S\.11.3]{watson1944bessel}), written in vector notation as
    \begin{align}
        \sum_{i+j=n} \mathrm{e}^{\mathrm{i}\left(i\,\mathrm{arg}(\mathbf{x}) + j\,\mathrm{arg}(\mathbf{y})\right)} J_{i}(|\mathbf{x}|) J_{\pm j}(|\mathbf{y}|) ={}& J_{n}(|\mathbf{x}\pm\mathbf{y}|) \,\mathrm{e}^{\mathrm{i}n\,\mathrm{arg}(\mathbf{x}\pm\mathbf{y})},\label{Bessel-gen:1}
    \end{align}
    for any $\mathbf{x},\mathbf{y}\in\mathbb{R}^2$.
\end{rmk}
In particular, fixing $j\in[1,m]$ and taking $\mathbf{x}_{i}=(-r,0)^{T}$ for $i\in[1, j]$ and $\mathbf{x}_{i}=(r,0)^{T}$ for $i\in[j+1,m]$, we obtain
    \begin{equation}
            J_{n}((m-2j)r) = \sum_{|\mathbf{k}|=n}\; \prod_{i=j+1}^{m}J_{k_{i}}(r)\prod_{\ell=1}^{j} J_{-k_{\ell}}(r).\label{ident:Besselsums-n}
    \end{equation}

\begin{rmk}
    Each convolutional sum can also be written with nested convolutional sums, and we note the following examples
\begin{subequations}\label{ident:Besselsums}
\begin{align}
    \delta_{n,0} ={}& \sum_{i+j=n}\; J_{i}(r) J_{-j}(r),\label{ident:Besselsums-0}\\
    J_{n}(r) ={}& \sum_{i+j+k=n}\; J_{i}(r)J_{j}(r)J_{-k}(r) = \sum_{i+j=n}\; J_{i}(2r)J_{-j}(r),\label{ident:Besselsums-1}\\
    J_{n}(2r) ={}& \sum_{i+j=n}\; J_{i}(r) J_{j}(r),\label{ident:Besselsums-2}\\
    J_{n}(3r) ={}& \sum_{i+j+k=n}\; J_{i}(r)J_{j}(r)J_{k}(r) = \sum_{i+j=n}\; J_{i}(2r)J_{j}(r)\label{ident:Besselsums-3}
\end{align}    
\end{subequations}
which we will use presently. Here, $\delta_{n,0}$ denotes the standard Kronecker delta, which satisfies $\delta_{i,j}=1$ if $i=j$ and $\delta_{i,j}=0$ otherwise. Clearly each summation is symmetric over $i,j,k\in\mathbb{Z}$, and so one can also permute the summation indices to obtain additional formulae.
\end{rmk}

The identities stated in \eqref{ident:Besselsums} are inherently connected to the Bessel expansion of the complex function
\begin{equation}\label{defn:u-rolls}
    u_{S}(r,\theta):= \mathrm{e}^{\mathrm{i}r\cos(\theta)} = \sum_{k\in\mathbb{Z}} \mathrm{i}^{k} J_{k}(r)\,\mathrm{e}^{\mathrm{i}k\theta}
\end{equation}
which we call a \emph{stripe} pattern, as plotted in Figure~\ref{fig:profiles}(a). We can also write down Bessel expansions for other two-dimensional patterns, which then allow us to derive additional nonlinear Bessel identities similar to \eqref{ident:Besselsums}. Due to the increased complexity of these identities, we restrict to only the quadratic and cubic cases.

%% Figure 3: Plots of dihedral patterns considered
\begin{figure}[t!]
    \centering
    \includegraphics[width=0.9\linewidth]{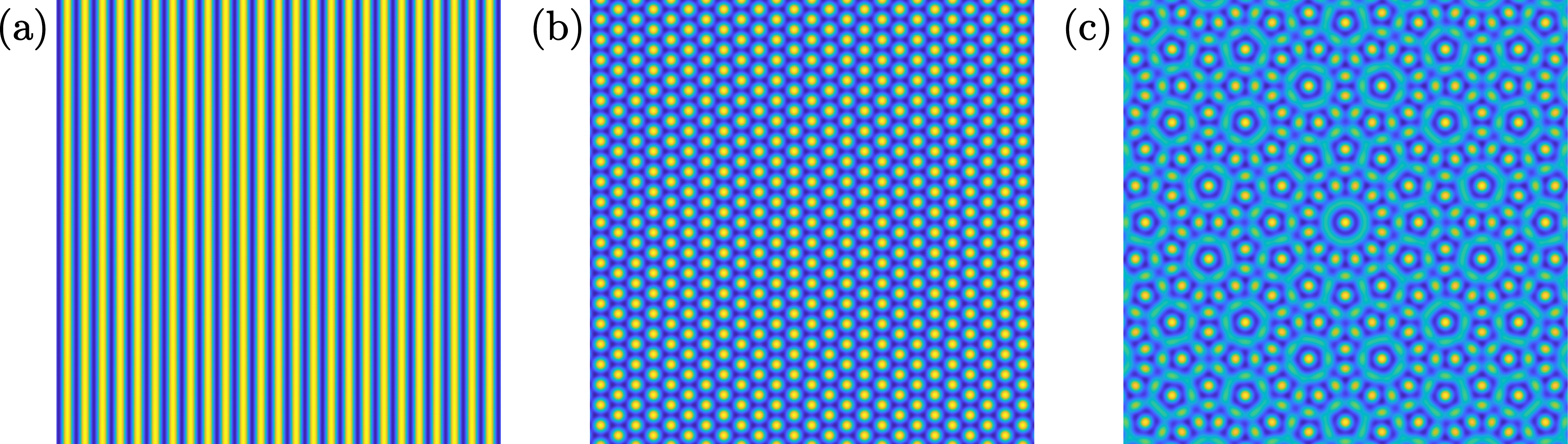}
    \caption{Surface plots of the real part of (a) $u_{S}$, (b) $u_{H}$, and (c) $u_{Q}$, defined in \eqref{defn:u-rolls}, \eqref{defn:u-hexagons}, and \eqref{defn:u-quasi}, respectively}
    \label{fig:profiles}
\end{figure}

We consider the following function 
\begin{equation}\label{defn:u-hexagons}
    u_{H}(r,\theta):= \frac{1}{3}\left(\mathrm{e}^{\mathrm{i}r\cos(\theta)} + \mathrm{e}^{\mathrm{i}r\cos(\theta + \frac{2\pi}{3})} + \mathrm{e}^{\mathrm{i}r\cos(\theta-\frac{2\pi}{3})} \right)
\end{equation}
which we refer to as a \emph{hexagon} pattern, as plotted in Figure~\ref{fig:profiles}(b). We note that $u_{H}$ has the following Bessel expansion
\begin{equation*}
    u_{H}(r,\theta) = \frac{1}{3}\sum_{k\in\mathbb{Z}} \mathrm{i}^{k} (1 + 2\cos(\tfrac{2k\pi}{3}))J_{k}(r)\,\mathrm{e}^{\mathrm{i}k\theta} = \sum_{k\in\mathbb{Z}} \mathrm{i}^{3k} J_{3k}(r)\,\mathrm{e}^{\mathrm{i}3k\theta}
\end{equation*}
and so we can derive the following result,
\begin{prop}\label{prop:hex}
Fix $n\in\mathbb{Z}$. Then,
\begin{equation}\label{ident:hex;Besselsums}
    \begin{split}
        \sum_{i+j=n} J_{3i}(r) J_{3j}(r) ={}& \frac{2}{3} J_{-3n}(r) + \frac{1}{3}J_{3n}(2r)\\
        \sum_{i+j=n} J_{3i}(r) J_{-3j}(r) ={}& \frac{1}{3}\delta_{n,0} + f_{1}(\sqrt{3}\,r)\\
        \sum_{i+j+k=n} J_{3i}(r) J_{3j}(r)J_{3k}(r) ={}& \frac{2}{9}\delta_{n.0} + \frac{1}{9} J_{3n}(3r) + f_{2}(\sqrt{3}\,r)\\
        \sum_{i+j+k=n} J_{3i}(r) J_{3j}(r)J_{-3k}(r) ={}& \frac{5}{9}J_{3n}(r) + \frac{2}{9} J_{-3n}(2r) + f_{3}(\sqrt{7}\,r)\\
    \end{split}
\end{equation}
where we collect terms with irrational frequencies $\omega_1,\dots,\omega_n$ into arbitrary functions $f_{i}(\omega_1\,r,\dots,\omega_n\,r)$ for $i=1,2,3$. 
\end{prop}
The proof of Proposition~\ref{prop:hex} (as well as subsequent propositions) follows a standard approach. Taking products of $\{u_{H},\overline{u_{H}}\}$, one obtains terms of the form $\mathrm{exp}(\mathrm{i}r[c_1\cos(\theta + m_1) + c_2\cos(\theta + m_2) + c_3 \cos(\theta + m_3)])$ for $c_i\in\{-1,0,1\}$, $m_i\in\{0,\frac{2\pi}{3}, - \frac{2\pi}{3}\}$. Using trigonometric identities, each term can then be expressed as $\mathrm{exp}(\mathrm{i}C r\cos(\theta + \phi))$, where $C,\phi$ can be found explicitly in terms of $\{c_1,c_2,c_3, m_1, m_2, m_3\}$. Following a projection onto the $n$-th angular Fourier mode $\mathrm{e}^{\mathrm{i}n\theta}$ and an application of the Hansen--Bessel formula \eqref{ident:Hansen-Bessel}, one obtains the results stated in \eqref{ident:hex;Besselsums}. The details are simple but laborious, and so we omit them from this work.

We next apply a translation $(x,y)\mapsto(x, y + \frac{2\pi}{\sqrt{3}})$ to \eqref{defn:u-hexagons}, resulting in 
\begin{equation}\label{defn:u-rhomb}
    u_{R}(r,\theta):= \frac{1}{3}\left(\mathrm{e}^{\mathrm{i}r\cos(\theta)} - \mathrm{e}^{\mathrm{i}r\cos(\theta + \frac{2\pi}{3})} - \mathrm{e}^{\mathrm{i}r\cos(\theta-\frac{2\pi}{3})} \right)
\end{equation}
which breaks the $\mathbb{D}_{6}$ symmetry of $u_{H}$. The resultant pattern is a hexagonal lattice with a depression at the origin; we note that $\mathbb{D}_{2}$ symmetry is preserved, and so we refer to $u_{R}$ as a \emph{rhombic} pattern. For the Bessel expansion of $u_{R}$, we find that
\begin{equation*}
    \begin{split}
    u_{R}(r,\theta)={} \sum_{k\in\mathbb{Z}} \mathrm{i}^{k}a_{k}J_{k}(r)\,\mathrm{e}^{\mathrm{i}k\theta},
    \end{split}
\end{equation*}
where $a_n:= \frac{1}{3}\left(1 -2 \cos(\tfrac{2k\pi}{3})\right)$,
and we obtain the following result,
\begin{prop}\label{prop:rhom}
Fix $n\in\mathbb{Z}$. Then,
\begin{equation}\label{ident:rhombic;Besselsums}
    \begin{split}
        \sum_{i+j=n} a_{i} a_{j} J_{i}(r) J_{j}(r) ={}& \frac{2}{3} a_{n} J_{-n}(r) + \frac{1}{9}\left(1 + 2\cos\left(\tfrac{2n\pi}{3}\right)\right)J_{n}(2r)\\
        \sum_{i+j=n} a_{i} a_{j} J_{i}(r) J_{-j}(r) ={}&\frac{1}{3}\delta_{n,0} + f_{1}(\sqrt{3}\,r)\\
        \sum_{i+j+k=n} a_{i} a_{j} a_{k} J_{i}(r) J_{j}(r) J_{k}(r) ={}&\frac{2}{9}\delta_{n.0} + \frac{1}{9} a_{n} J_{n}(3r) + f_{2}(\sqrt{3}\,r)\\
        \sum_{i+j+k=n} a_{i} a_{j} a_{k} J_{i}(r) J_{j}(r) J_{-k}(r) ={}& \frac{5}{9} a_{n} J_{n}(r) + \frac{2}{27}\left(1 + 2\cos(\tfrac{2n\pi}{3})\right)J_{-n}(2r) + f_{3}(\sqrt{7}\,r)\\
    \end{split}
\end{equation}
where we collect terms with irrational frequencies $\omega_1,\dots,\omega_n$ into arbitrary functions $f_{i}(\omega_1\,r,\dots,\omega_n\,r)$ for $i=1,2,3$. 
\end{prop}
Finally, we consider the superposition of two copies of $u_{H}$ rotated by an angle of $\frac{\pi}{12}$ and $-\frac{\pi}{12}$, respectively, which we write as
\begin{equation}\label{defn:u-quasi}
    u_{Q}(r,\theta):= \frac{1}{6}\left(\mathrm{e}^{\mathrm{i}r\cos(\theta + \frac{\pi}{12})} - \mathrm{e}^{\mathrm{i}r\cos(\theta + \frac{9\pi}{12})} - \mathrm{e}^{\mathrm{i}r\cos(\theta-\frac{7\pi}{12})} + \mathrm{e}^{\mathrm{i}r\cos(\theta-\frac{\pi}{12})} - \mathrm{e}^{\mathrm{i}r\cos(\theta + \frac{7\pi}{12})} - \mathrm{e}^{\mathrm{i}r\cos(\theta-\frac{9\pi}{12})} \right).
\end{equation}
The function $u_{Q}$ is a special case of the arbitrarily rotated hexagon lattices considered in \cite{IoossRucklidge2022RotatedHexagons} and forms a quasipattern with $\mathbb{D}_{12}$ symmetry, as plotted in Figure~\ref{fig:profiles}(c). The \emph{quasipattern} $u_{Q}$ has the following Bessel expansion
\begin{equation*}
    \begin{split}
        u_{Q}(r,\theta) = \sum_{k\in\mathbb{Z}} \mathrm{i}^{3k} b_k J_{3k}(r)\,\mathrm{e}^{\mathrm{i}3k\theta}
    \end{split}
\end{equation*}
where $b_k:=\cos(\frac{k\pi}{4})$, and so we derive our final set of identities,
\begin{prop}\label{prop:D12}
Fix $n\in\mathbb{Z}$. Then,
\begin{equation}\label{ident:D12;Besselsums}
    \begin{split}
        \sum_{i+j=n} b_{i} b_{j} J_{3i}(r) J_{3j}(r) 
        ={}& \frac{1}{3}b_{n}J_{-3n}(r) + \frac{1}{6} b_{n}J_{3n}(2r) + f_{1}\left(\sqrt{2}\,r,\sqrt{2\pm\sqrt{3}}\,r\right)\\
        \sum_{i+j=n} b_{i} b_{j} J_{3i}(r) J_{-3j}(r) 
        ={}& \frac{1}{6}\delta_{n,0} + f_{2}\left(\sqrt{2}\,r, \sqrt{3}\,r, \sqrt{2\pm\sqrt{3}}\,r\right)\\
        \sum_{i+j+k=n} b_{i} b_{j} b_{k} J_{3i}(r) J_{3j}(r) J_{3k}(r) ={}& \frac{1}{18}\delta_{n,0} + \frac{1}{36}b_n J_{3n}(3 r) + f_{3}\left(\sqrt{2}\,r,\sqrt{3}\,r,\sqrt{2\pm\sqrt{3}}\,r\right)\\
        & + f_{4}\left(\sqrt{5}\,r,\sqrt{5\pm2\sqrt{3}}\,r\right) \\
        \sum_{i+j+k=n} b_{i} b_{j} b_{k} J_{3i}(r) J_{3j}(r) J_{-3k}(r) ={}& \frac{11}{36} b_{n} J_{3n}(r) + \frac{1}{18}b_n J_{-3n}(2r) + {\textstyle f_{5}\left(\sqrt{2}\,r, \sqrt{2\pm\sqrt{3}}\,r\right)}\\
        & + {\textstyle f_{6}\left(\sqrt{5}\,r,\sqrt{7}\,r,\sqrt{4\pm2\sqrt{3}}\,r,\sqrt{4\pm\sqrt{3}}\,r,\sqrt{5\pm2\sqrt{3}}\,r\right)} \\
    \end{split}
\end{equation}
where we collect terms with irrational frequencies $\omega_1,\dots,\omega_n$ into arbitrary functions $f_{i}(\omega_1\,r,\dots,\omega_n\,r)$ for $i=1,\dots,6$. 
\end{prop}

Other more complicated patterns can be considered beyond the four that we concentrate on in this paper; for instance one could choose squares or other cases of rotated hexagons considered in \cite{IoossRucklidge2022RotatedHexagons}, including other quasipattern or superlattice structures. In the following remark we highlight an extension of Proposition \ref{prop:D12} to arbitrary rotations of two hexagon lattices.

\begin{rmk}\label{rmk:rotate}
    Consider the pattern $u_{\alpha}(r,\theta) = \frac{1}{2}\left( u_{H}(r,\theta+\alpha) + u_{H}(r,\theta-\alpha)\right)$ for some $\alpha\in(0,\frac{\pi}{6})$. Then,
    \begin{equation*}
        u_{\alpha}(r,\theta) = \sum_{k\in\mathbb{Z}} \mathrm{i}^{3k} c_{k} J_{3k}(r)\,\mathrm{e}^{\mathrm{i}3k\theta}
    \end{equation*}
    where $c_k:=\cos(3k\alpha)$, and \eqref{ident:D12;Besselsums} is satisfied by replacing any $b_n$ terms by $c_{n}$ and changing the irrational frequencies. In particular, the irrational frequencies $\omega_1,\dots,\omega_n$ within each arbitrary function $f_i$ in \eqref{ident:D12;Besselsums} all depend on $\alpha$ but do not equal 1 for any values of $\alpha\in(0,\frac{\pi}{6})$. Hence, an amplitude equation derived for $u_Q$ is also an amplitude equation for $u_{\alpha}$ after replacing any $b_n$ terms by $c_n$ for any $\alpha\in(0,\frac{\pi}{6})$.
\end{rmk}
%
%
%
%
%
%%%%%%%%%%%%%%%%%
%   Section 3   %
%%%%%%%%%%%%%%%%%   
%
%
%
%
%
\section{Radial amplitude equation for the Swift--Hohenberg equation}\label{s:amplitude}
We now wish to apply the tools presented in the previous section in order to derive a radial amplitude equation for a planar PDE \eqref{e:SHE}. We define $\mu = \varepsilon^2 \hat{\mu}$ for $0<\varepsilon\ll1$, and introduce the slow variables $T:=\varepsilon^2 t$, $R:=\varepsilon r$, with $k$-index slow differential operators $\partial_{R}$ and $\hat{\Delta}_{k}$. 

We recall from \S\ref{ss:bessel-deriv} that, under the assumption that $u_n$ has the form $u_n(T,r,R)$, where $u_n$ acts like an $J_n$ Bessel function with respect to $r$ and like an axisymmetric function with respect to $R$, we can express the differential operator for \eqref{e:SHE-n} in the asymptotic form \eqref{diff:asymp},
\begin{equation*}
    \begin{split}
        \left(1 + \Delta_{n}\right)^2 u_n ={}& \left(1 + \Delta_{n}\right)^{2}u_n + \varepsilon\,\left[4\,\partial_{r}\left(1 + \Delta_{n}\right)\partial_{R} u_n\right] + \varepsilon^2 \left[2\left(1 + \Delta_{n}\right) \hat{\Delta}_{0} u_n + 4\,\Delta_{n}\,\partial_{R}^{2} u_n\right] + \mathcal{O}(\varepsilon^3).\\
    \end{split}
\end{equation*}
Then, taking a regular perturbation
\begin{equation*}
    u_n = \varepsilon v_{n}^{(0)}(T,r,R) + \varepsilon^2 v_{n}^{(1)}(T,r,R) + \varepsilon^3 v_{n}^{(2)}(T,r,R) + \mathcal{O}(\varepsilon^{4}), 
\end{equation*}
we obtain
\begin{subequations}\label{eqn:asymp}
\begin{align}
    \mathcal{O}(\varepsilon^{1})& & \quad 0 ={}& -\left(1 + \Delta_{n}\right)^{2}v^{(0)}_n \label{eqn:asymp-1}\\
    \mathcal{O}(\varepsilon^{2})& & \quad 0 ={}& -\left(1 + \Delta_{n}\right)^{2}v^{(1)}_n - 4\partial_{r}\left(1 + \Delta_{n}\right)\partial_{R} v^{(0)}_n + \nu \sum_{i+j=n}v^{(0)}_{i}v^{(0)}_{j}\label{eqn:asymp-2}\\
    \mathcal{O}(\varepsilon^{3})& & \quad \partial_T v^{(0)}_n ={}& -\left(1 + \Delta_{n}\right)^{2}v^{(2)}_n - 4\partial_{r}\left(1 + \Delta_{n}\right)\partial_{R}v^{(1)}_n - 2\left(1 + \Delta_{n}\right)\hat{\Delta}_{0} v^{(0)}_n\nonumber\\
    & & &\qquad  - 4\,\Delta_{n}\,\partial_{R}^{2}v^{(0)}_n - \hat{\mu} v^{(0)}_n + 2\nu \sum_{i+j=n}v^{(0)}_{i}v^{(1)}_{j} - \sum_{i+j+k=n}v^{(0)}_{i}v^{(0)}_{j}v^{(0)}_{k} \label{eqn:asymp-3}
\end{align}
\end{subequations}
The rest of this section is dedicated to solving \eqref{eqn:asymp} by applying the identities introduced in \S\ref{ss:bessel-nonlin}. We choose particular radial profiles that correspond to the patterns $u_{S}, u_{H}, u_{R}, u_{Q}$ defined in \eqref{defn:u-rolls}, \eqref{defn:u-hexagons}, \eqref{defn:u-rhomb} and \eqref{defn:u-quasi}, respectively.

%%%%%%%%%%%%%%%%%%%%%%%%%%%%%%%%%%%%%%%%%%%%
\subsection{Fully localised stripes}\label{ss:amplitude-rolls}
For the stripe pattern $u_R$ introduced in \S\ref{ss:bessel-nonlin}, we solve \eqref{eqn:asymp-1} with the ansatz
\begin{equation}\label{sol:v0}
 v_n^{(0)} = A(T,R) J_{n}(r) + \overline{A}(T,R)J_{-n}(r),
\end{equation}
where we note that we have chosen the complex amplitude $A(T,R)$ to be identical for all values of $n\in\mathbb{Z}$. We recall that this is equivalent to assuming that the original function $u(t,x,y)$ is composed of a planar pattern multiplied with a slowly varying axisymmetric amplitude. Using this ansatz, \eqref{eqn:asymp-2} becomes 
\begin{equation*}
\begin{aligned}
     \left(1 + \Delta_{n}\right)^{2}v^{(1)}_n ={}& \nu \sum_{i+j=n}v^{(0)}_{i}v^{(0)}_{j}\\
     ={}& \nu \sum_{i+j=n}\left(A J_{i}(r) + \overline{A}J_{-i}(r)\right)\left(A J_{j}(r) + \overline{A}J_{-j}(r)\right)\\
     ={}& \nu A^2 \sum_{i+j=n} J_{i}(r) J_{j}(r) + 2 \nu |A|^2\sum_{i+j=n}J_{i}(r)J_{-j}(r) + \nu \overline{A}^2\sum_{i+j=n} J_{-i}(r)J_{-j}(r)\\
     ={}& \nu A^2 J_{n}(2 r) + 2 \nu |A|^2 \delta_{n,0} + \nu \overline{A}^2 J_{-n}(2r)
\end{aligned}
\end{equation*}
Taking an ansatz of the form 
\begin{equation*}
    v_{n}^{(1)} = \delta_{n,0}\, O_{1} + A_1 J_{n}(r) + \overline{A_{1}}J_{-n}(r) + 
    B_1 J_{n}(2r) + \overline{B_{1}}J_{-n}(2r)
\end{equation*}
    we obtain
\begin{equation*}
\begin{aligned}
     \delta_{n,0}\, O_{1} + 
    9 B_1 J_{n}(2r) + 9 \overline{B_{1}}J_{-n}(2r) ={}& \nu A^2 J_{n}(2 r) + 2 \nu |A|^2 \delta_{n,0} + \nu \overline{A}^2 J_{-n}(2r)
\end{aligned}
\end{equation*}
which implies that
\begin{equation}\label{sol:v1}
    v_{n}^{(1)} = 2 \nu |A|^2 \delta_{n,0} + A_1 J_{n}(r) + \overline{A_{1}}J_{-n}(r) + 
    \frac{\nu}{9}A^2 J_{n}(2r) + \frac{\nu}{9}\overline{A}^2J_{-n}(2r).
\end{equation}
Then, we substitute \eqref{sol:v0} and \eqref{sol:v1} into \eqref{eqn:asymp-3} to obtain
\begin{equation*}
\begin{aligned}
    \partial_{T} A J_{n}(r) + \partial_{T}\overline{A} J_{-n}(r) ={}& -\left(1 + \Delta_{n}\right)^{2}v^{(2)}_n \\
    &\qquad + 12\left[\frac{\nu}{9}\partial_{R}\left(A^2\right) \partial_{r}J_{n}(2r) +\frac{\nu}{9}\partial_{R}\left(\overline{A}^2\right) \partial_{r} J_{-n}(2r)\right]\\
    & \qquad + \left[4 \partial_{R}^{2}A - \hat{\mu} A + 4\left(\frac{19\nu^2}{18} - \frac{3}{4}\right)|A|^2 A \right]J_{n}(r) \\
    & \qquad + \left[4 \partial_{R}^{2}\overline{A} - \hat{\mu} \overline{A} + 4\left(\frac{19\nu^2}{18} - \frac{3}{4}\right)|A|^2 \overline{A} \right]J_{-n}(r) \\
    & \qquad + 2\nu \left(A \overline{A_{1}} + \overline{A} A_1\right)\delta_{n,0} \\
    & \qquad + 2\nu A A_1 J_{n}(2r) + 2\nu \overline{A} \overline{A_{1}} J_{-n}(2r) \\
    & \qquad + \left(\frac{4\nu^2}{18} - 1 \right) A^3 J_{n}(3r)  + \left(\frac{4\nu^2}{18} -  1\right)\overline{A}^{3} J_{-n}(3r) \\
\end{aligned}
\end{equation*}
 We require the $A$-dependent terms to lie in the range of the operator $(1+\Delta_{n})^2$, and so the coefficients of the prime frequency $J_{\pm n}(r)$ must be identically zero. Hence, we obtain a solvability condition for $A$ in the form of the standard cubic Ginzburg--Landau amplitude equation
\begin{equation}\label{e:Ginz}
    \partial_T A = 4 \partial_{R}^2 A - \hat{\mu}A + 4\left(\frac{19\nu^2}{18} - \frac{3}{4}\right)|A|^2 A,
\end{equation}
now posed on the half-line $R\geq0$ rather than on the whole real line. We note that the equation \eqref{e:Ginz} is identical to the amplitude equation \eqref{e:Ginz-1D} derived in the one-dimensional Swift--Hohenberg equation. 

In order to find a steady localised solution of \eqref{e:Ginz}, we restrict to the invariant subspace $A\in\mathbb{R}$. The assumption that $A$ is axisymmetric induces a compatibility condition at $R=0$ of the form $\partial_{R}A(0) = 0$. By reversibility of the time-independent case of \eqref{e:Ginz}, any steady localised solution of \eqref{e:Ginz} must thus form a homoclinic orbit. Then, for $\nu\leq \sqrt{\frac{27}{38}}$ there are no real steady localised solutions of \eqref{e:Ginz}, and for $\nu> \sqrt{\frac{27}{38}}$ there is a real steady localised solution
\begin{equation}\label{sol:A}
A(R) = \pm \sqrt{\frac{18\hat{\mu}}{38\nu^2 - 27}} \,\mathrm{sech}\left(\frac{\sqrt{\hat{\mu}}}{2}R\right).
\end{equation}

Having found a real steady localised solution to \eqref{e:Ginz}, we can write down a steady fully localised solution to \eqref{e:SHE} of the form
\begin{equation*}
\begin{split}
    \tilde{u}(r,\theta) ={}& \varepsilon A(\varepsilon r)\sum_{n\in\mathbb{Z}} \left( J_{n}(r) + J_{-n}(r)\right)\mathrm{e}^{\mathrm{i}n\theta} + \mathcal{O}(\varepsilon^2),\\
    ={}& 2\varepsilon A(\varepsilon r)\sum_{n\in\mathbb{Z}} J_{|2n|}(r) \cos(2n\theta) + \mathcal{O}(\varepsilon^2),\\
    ={}& 2\varepsilon A(\varepsilon r)\cos(r\sin(\theta)) + \mathcal{O}(\varepsilon^2),\\
\end{split}
\end{equation*}
where we have used the real-valued Jacobi--Anger expansion
\begin{equation*}
    \cos(r\sin(\theta)) = \sum_{n\in\mathbb{Z}} J_{|2n|}(r)\,\cos(2n\theta)
\end{equation*}
and the fact that $J_{-n}(r) = (-1)^n J_{n}(r)$ for any $n\in\mathbb{Z}$.
Converting into Cartesian coordinates, we have found a fully localised solution to \eqref{e:SHE} of the form
\begin{equation*}
    u(x,y) = 2\varepsilon A\left(\varepsilon \sqrt{x^2+y^2}\right)\cos(y) + \mathcal{O}(\varepsilon^2),
\end{equation*}
as well as 
\begin{equation*}
    u(x,y) = 2\varepsilon A\left(\varepsilon \sqrt{x^2+y^2}\right)\cos(x) + \mathcal{O}(\varepsilon^2),
\end{equation*}
since the SHE \eqref{e:SHE} is invariant under rotation.

%%%%%%%%%%%%%%%%%%%%%%%%%%%%%%%%%%%%%%%%%%%%
\subsection{Other fully localised patterns}\label{ss:amplitude-hex}
We now consider the other dihedral patterns $u_{H}$, $u_{R}$ and $u_{Q}$ defined in \S~\ref{ss:bessel-nonlin}. We note that the quadratic projections in \eqref{ident:hex;Besselsums}, \eqref{ident:rhombic;Besselsums} \& \eqref{ident:D12;Besselsums} all include $J_{n}(r)$ terms; this means that the quadratic terms in \eqref{eqn:asymp} do not lie entirely in the range of $(1+\Delta_{n})^2$ and thus cannot be removed by normal form transformations. We note that this is not unique to our radial projections, and is also observed for hexagons in Cartesian coordinates due to the presence of resonant triads in Fourier space (see Section 5.4 in \cite{Hoyle2006}, for example).

To circumvent this problem, we take $\nu:=\varepsilon \hat{\nu}$ such that \eqref{eqn:asymp} becomes
\begin{subequations}\label{eqn:asymp;hex}
\begin{align}
    \mathcal{O}(\varepsilon^{1})& & \quad 0 ={}& -\left(1 + \Delta_{n}\right)^{2}v^{(0)}_{n} \label{eqn:asymp;hex-1}\\
    \mathcal{O}(\varepsilon^{2})& & \quad 0 ={}& -\left(1 + \Delta_{n}\right)^{2}v^{(1)}_{n} - 4\partial_{r}\left(1 + \Delta_{n}\right)\partial_{R} v^{(0)}_{n}\label{eqn:asymp;hex-2}\\
    \mathcal{O}(\varepsilon^{3})& & \quad \partial_T v^{(0)}_{n} ={}& -\left(1 + \Delta_{n}\right)^{2}v^{(2)}_{n} - 4\partial_{r}\left(1 + \Delta_{n}\right)\partial_{R}v^{(1)}_{n} - 2\left(1 + \Delta_{n}\right)\hat{\Delta}_{0} v^{(0)}_{n}\nonumber\\
    & & &\qquad  - 4\,\Delta_{n}\,\partial_{R}^{2}v^{(0)}_{n} - \hat{\mu} v^{(0)}_{n}  + \hat{\nu} \sum_{i+j=n}v^{(0)}_{i}v^{(0)}_{j} - \sum_{i+j+k=n}v^{(0)}_{i}v^{(0)}_{j}v^{(0)}_{k} \label{eqn:asymp;hex-3}
\end{align}
\end{subequations}
For the hexagonal pattern $u_{H}$ we solve \eqref{eqn:asymp;hex-1} and \eqref{eqn:asymp;hex-2} with the ansatz
\begin{equation}\label{sol:hex1}
    v_{3n}^{(0)} = A J_{3n}(r) + \overline{A} J_{-3n}(r),  \qquad v_{3n}^{(1)} = A_1 J_{3n}(r) + \overline{A_1} J_{-3n}(r)
\end{equation}
and $v_{n}^{(0)}=v_{n}^{(1)}=0$ for all $n\in\mathbb{Z}$ with $n$ not a multiple of $3$, such that \eqref{eqn:asymp;hex-3} becomes
\begin{equation*}
    \begin{aligned}
        \partial_T A J_{3n}(r) + \partial_{T}\overline{A} J_{-3n}(r) ={}& 4\,\partial_{R}^{2}A J_{3n}(r) + 4\,\partial_{R}^{2}\overline{A} J_{-3n}(r) - \hat{\mu} A J_{3n}(r) -  \hat{\mu} \overline{A} J_{-3n}(r) \\
        & \qquad -\left(1 + \Delta_{3n}\right)^{2}v^{(2)}_{3n} + 2\hat{\nu} |A|^2 \sum_{i+j=n} J_{3i}(r)J_{-3j}(r)\\
        &\qquad  + \hat{\nu} A^{2} \sum_{i+j=n} J_{3i}(r) J_{3j}(r)  + \hat{\nu}\overline{A}^2 \sum_{i+j=-n}J_{3i}(r)J_{3j}(r)\\
        &\qquad - A^3 \sum_{i+j+k=n}J_{3i}(r)J_{3j}(r) J_{3k}(r) - 3|A|^2 A \sum_{i+j+k=n} J_{3i}(r) J_{3j}(r) J_{-3k}(r)\\
        &\qquad - 3 |A|^2 \overline{A}\sum_{i+j+k=-n} J_{3i}(r) J_{3j}(r) J_{-3k}(r) - \overline{A}^3 \sum_{i+j+k=-n} J_{3i}(r) J_{3j}(r)J_{3k}(r)
    \end{aligned}
\end{equation*}
and, using the identities \eqref{ident:hex;Besselsums},
\begin{equation*}
    \begin{aligned}
        \partial_T A J_{3n}(r) + \partial_{T}\overline{A} J_{-3n}(r) ={}& \left[4\,\partial_{R}^{2}A - \hat{\mu} A  - \frac{5}{3}|A|^2 A  + \frac{2\hat{\nu}}{3}\overline{A}^2 \right]J_{3n}(r) -\left(1 + \Delta_{3n}\right)^{2}v^{(2)}_{3n} \\
        &\qquad + \left[4\,\partial_{R}^{2}\overline{A} -  \hat{\mu}\overline{A} - \frac{5}{3}|A|^2 \overline{A} + \frac{2\hat{\nu}}{3} A^{2} \right] J_{-3n}(r)\\
        & \qquad + \frac{1}{3}\left[\hat{\nu} A^{2} - 2|A|^2 \overline{A} \right] J_{3n}(2r) + \frac{1}{3}\left[\hat{\nu} \overline{A}^2 - 2|A|^2 A\right] J_{-3n}(2r) \\
        &\qquad + \frac{2\hat{\nu}}{3}|A|^2 \delta_{n,0} - \frac{1}{9} A^3 J_{3n}(3r) - \frac{1}{9}\overline{A}^3 J_{-3n}(3r)\\
        &\qquad + f(\sqrt{3}\,r,\sqrt{7}\,r)\\
    \end{aligned}
\end{equation*}
for some function $f$ of irrational frequencies $\{\sqrt{3}\,r, \sqrt{7}\,r\}$. We define $A_h:=\frac{1}{3}A$ and again obtain a solvability condition for $A_h$ in the form of the amplitude equation
\begin{equation}\label{e:Ginz-hex}
    \begin{aligned}
        \partial_T A_h ={}&  4\,\partial_{R}^{2}A_h - \hat{\mu} A_h - 15|A_h|^2 A_h + 2\hat{\nu}\overline{A_h}^2.\\
    \end{aligned}
\end{equation}
Note, if we consider the rhombic pattern $u_{R}$ defined in \eqref{defn:u-rhomb} and follow the same approach as above, we obtain
\begin{equation}\label{sol:rhomb1}
    v_{n}^{(0)} = 3A_{h} a_{n}J_{n}(r) + 3\overline{A_{h}} a_{n} J_{-n}(r),  \qquad v_{n}^{(1)} = A_1 a_{n}J_{n}(r) + \overline{A_1} a_{n}J_{-n}(r)
\end{equation}
for all $n\in\mathbb{Z}$, where $A_{h}$ again satisfies \eqref{e:Ginz-hex}. Finally, considering the quasipattern $u_{Q}$ defined in \eqref{defn:u-quasi} and following the same approach, we obtain
\begin{equation}\label{sol:D12}
    v_{3n}^{(0)} = 6A_{q} b_{n}J_{3n}(r) + 6\overline{A_{q}} b_{n} J_{-3n}(r),  \qquad v_{3n}^{(1)} = A_1 b_{n}J_{3n}(r) + \overline{A_1} b_{n}J_{-3n}(r)
\end{equation}
and $v_{n}^{(0)}=v_{n}^{(1)}=0$ for all $n\in\mathbb{Z}$ with $n$ not a multiple of $3$, where $A_{q}$ satisfies
\begin{equation}\label{e:Ginz-D12}
    \begin{aligned}
        \partial_T A_q ={}&  4\,\partial_{R}^{2}A_q - \hat{\mu} A_q - 33|A_q|^2 A_q + 2\hat{\nu}\overline{A_q}^2.\\
    \end{aligned}
\end{equation}

%% Figure 4: Plots of amplitude equations when approaching the Maxwell point
\begin{figure}[t!]
    \centering
    \includegraphics[width=0.9\linewidth]{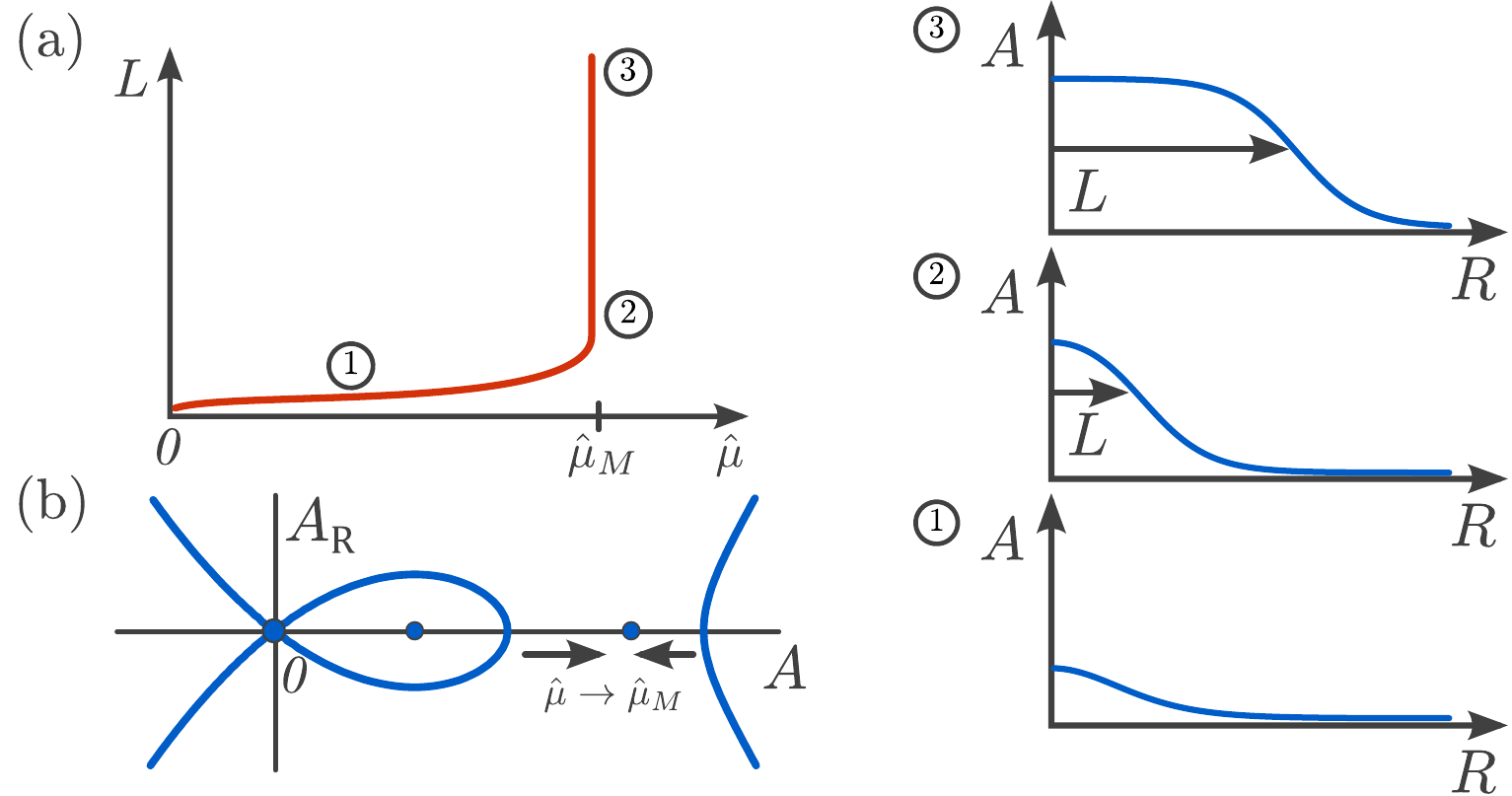}
    \caption{Sketches of (a) the bifurcation diagram and (b) the corresponding phase plane of the amplitude equations (\ref{e:Ginz-hex}) and (\ref{e:Ginz-D12}). A homoclinic orbit bifurcates from $\hat{\mu}=0$ and tends to a heteroclinic connection as $\hat{\mu}\to\hat{\mu}_{M}$, resulting in a localised solution that widens into a uniform state as $\hat{\mu}$ increases.}
    \label{fig:maxwell_point}
\end{figure}
 
Since \eqref{e:Ginz-D12} is independent of $b_n$, it follows from Remark \ref{rmk:rotate} that \eqref{e:Ginz-D12} is also the amplitude equation for the superposition of two arbitrarily rotated hexagon lattices $u_{\alpha}$ for any rotation angle $\alpha\in(0,\frac{\pi}{6})$. In order to study steady localised solutions of \eqref{e:Ginz-hex} and \eqref{e:Ginz-D12}, we again restrict to the real invariant subspaces $A_{h}\in\mathbb{R}$, $A_{q}\in\mathbb{R}$. Stationary, real solutions of equations~(\ref{e:Ginz-hex}) and (\ref{e:Ginz-D12}) can be investigated using phase plane analysis, where we observe a small amplitude homoclinic solution that bifurcates off the trivial state and broadens to a ``table-top" localised pattern as one approaches the Maxwell point, $\hat{\mu}=\hat{\mu}_M$; see Figure~\ref{fig:maxwell_point}. In particular, for general equations of the form
\begin{equation*}
    0=4\partial_{R}^{2}A - \hat{\mu}A + 2\hat{\nu}A^2 - a A^3
\end{equation*}
there is an explicit homoclinic solution
\begin{equation*}
\begin{split}
    A(R) ={}& \frac{\sqrt{\frac{2}{a}}\,\hat{\mu}}{\sqrt{\hat{\mu}_{M}} + \sqrt{\hat{\mu}_{M} - \hat{\mu}}\,\cosh\left(\frac{\sqrt{\hat{\mu}}}{2}R\right)},\\
\end{split}
\end{equation*}
for all $0<\hat{\mu}<\hat{\mu}_{M}$, where the Maxwell point $\hat{\mu}_{M}$ is located at $\hat{\mu}_{M} = \frac{8}{9a}\hat{\nu}^2$. Then, for the hexagons and rhomboids we find that $\hat{\mu}_{M} = \hat{\mu}_{H} := \frac{8}{135}\hat{\nu}^2$, which is exactly the same Maxwell point as found for fully localised hexagons in \cite{Lloyd2008LocalizedHexagons}, and for the quasipatterns $\hat{\mu}_{M} = \hat{\mu}_{Q} := \frac{8}{297}\hat{\nu}^2<\hat{\mu}_H$.

Hence, for $\hat{\mu}<\hat{\mu}_{H}$ there is a homoclinic orbit that corresponds to a localised solution $A_h$ to \eqref{e:Ginz-hex}, and thus to steady localised solutions $u_{H}, u_{R}$ to \eqref{e:SHE} of the form
\begin{equation*}
\begin{split}
    u_{H}(x,y) ={}& 2 \varepsilon A_h(\varepsilon \sqrt{x^2 + y^2})\left(\cos(x) + \cos(\tfrac{x-\sqrt{3}y}{2}) + \cos(\tfrac{x+\sqrt{3}y}{2})\right) + \mathcal{O}(\varepsilon^2),\\
    u_{R}(x,y) ={}& 2 \varepsilon A_h(\varepsilon \sqrt{x^2 + y^2})\left(\cos(x) - \cos(\tfrac{x-\sqrt{3}y}{2}) - \cos(\tfrac{x+\sqrt{3}y}{2})\right) + \mathcal{O}(\varepsilon^2).\\
\end{split}
\end{equation*}

Similarly, for $\hat{\mu}<\hat{\mu}_{Q}$ there is a homoclinic orbit that corresponds to a localised solution $A_q$ to \eqref{e:Ginz-D12}, and thus a steady localised solution $u_Q$ to \eqref{e:SHE} of the form
\begin{equation*}
\begin{split}
    u_{Q}(x,y) = 2 \varepsilon A_q(\varepsilon \sqrt{x^2 + y^2})\bigg(&\cos(\tfrac{(1+\sqrt{3})x+(1-\sqrt{3})y}{2\sqrt{2}}) + \cos(\tfrac{(1+\sqrt{3})x-(1-\sqrt{3})y}{2\sqrt{2}}) + \cos(\tfrac{(1-\sqrt{3})x+(1+\sqrt{3})y}{2\sqrt{2}})\\
    &\quad + \cos(\tfrac{(1-\sqrt{3})x-(1+\sqrt{3})y}{2\sqrt{2}}) + \cos(\tfrac{x + y}{\sqrt{2}}) + \cos(\tfrac{x - y}{\sqrt{2}})\bigg) + \mathcal{O}(\varepsilon^2).
\end{split}
\end{equation*}

\section{Fully localised hexagons in reaction--diffusion systems}\label{s:RD}
While we have so far only considered fully localised patterns in the Swift--Hohenberg equation \eqref{e:SHE}, our approach can also be applied to other PDE systems. To highlight this, we consider a general class of two-component reaction--diffusion systems and derive a radial amplitude equation corresponding to fully localised hexagons. Two-component reaction--diffusion systems are used to model a myriad of physical phenomena, ranging from dryland vegetation~\cite{vonHardenberg2001Vegetation} to nonlinear optics~\cite{LugiatoLefever1987Optics} to the spread of epidemics~\cite{Allen2008RDEpidemic}, and can be written in the following form
\begin{equation}\label{e:RD}
    \partial_t\mathbf{u} = \mathbf{D}(\mu)\Delta\mathbf{u} - \mathbf{f}(\mathbf{u};\mu)
\end{equation}
where $\mathbf{u}= \mathbf{u}(t,x,y)\in\mathbb{R}^2$ and $\mu\in\mathbb{R}$. We assume the diffusion matrix $\mathbf{D}(\mu)$ is invertible for all $\mu\in[0,\mu_0]$, for some $\mu_0>0$, and that the trivial state undergoes a pattern-forming instability at $\mu=0$. Following the work of Hill et al.~\cite{Hill2023DihedralSpots,Hill2024DihedralRings}, we consider stationary solutions of \eqref{e:RD} and invert the matrix $\mathbf{D}(\mu)$. Then, performing a Taylor expansion of \eqref{e:RD} about $(\mathbf{u},\mu) = (\mathbf{0},0)$, we obtain
\begin{equation}\label{e:RD-Taylor}
    \mathbf{0} = \Delta\mathbf{u} - \mathbf{M}_{1}\mathbf{u} - \mu\mathbf{M}_{2}\mathbf{u} - \nu \mathbf{Q}(\mathbf{u},\mathbf{u}) - \mathbf{C}(\mathbf{u},\mathbf{u},\mathbf{u}),
\end{equation}
where $\nu\in\mathbb{R}$ is a fixed parameter, $\mathbf{M}_i$ are fixed real matrices, and $\mathbf{Q}$, $\mathbf{C}$ are bilinear and trilinear functions, respectively. Note that we have truncated the Taylor expansion of \eqref{e:RD} at the cubic order in $\mathbf{u}$; any higher order terms do not affect our analysis, and so we omit them for simplicity. Finally, we assume that $\mathbf{M}_{1}$ has repeated eigenvalues $\lambda=-k_c^2$ with generalised eigenvectors $\hat{U}_0,\hat{U}_1\in\mathbb{R}^{2}$ such that
\begin{equation*}
    \left(\mathbf{M}_1 + k_c^2\right)\hat{U}_0 = \mathbf{0},\qquad \left(\mathbf{M}_1 + k_c^2\right)\hat{U}_1 = k_c^2\hat{U}_{0},
\end{equation*}
with dual vectors $\{\hat{U}_0^*, \hat{U}_1^*\}$ such that $\hat{U}_i^*\cdot\hat{U}_{j} = \delta_{i,j}$.

We restrict ourselves to looking for fully localised hexagon solutions to \eqref{e:RD-Taylor}, and so we expand $\mathbf{u}$ in the following $\mathbb{D}_{6}$ Fourier series
\begin{equation*}
    \mathbf{u}(x,y) = \tilde{\mathbf{u}}(r,\theta) = \sum_{n\in\mathbb{Z}}\mathbf{u}_{3n}(r)\mathrm{e}^{\mathrm{i}3n\theta}
\end{equation*}
with the standard reality condition $\mathbf{u}_{-3n} = \overline{\mathbf{u}_{3n}}$, for which \eqref{e:RD-Taylor} becomes
\begin{equation}\label{e:RD;Taylor-n}
    \mathbf{0} = \Delta_{3n}\mathbf{u}_{3n} - \mathbf{M}_{1}\mathbf{u}_{3n} - \mu\mathbf{M}_{2}\mathbf{u}_{3n} - \nu\sum_{i+j=n} \mathbf{Q}(\mathbf{u}_{3i},\mathbf{u}_{3j}) - \sum_{i+j+k=n} \mathbf{C}(\mathbf{u}_{3i},\mathbf{u}_{3j},\mathbf{u}_{3k})
\end{equation}
for all $n\in\mathbb{Z}$. We define $\mu = \varepsilon^2 \hat{\mu}$, $\nu = \varepsilon\hat{\nu}$ and $R = \varepsilon r$, and take an expansion of the form
\begin{equation*}
    \mathbf{u}_{3n} = \varepsilon\,\mathbf{u}_{3n}^{(0)}(r,R) + \varepsilon^2\,\mathbf{u}_{3n}^{(1)}(r,R) + \varepsilon^3\,\mathbf{u}_{3n}^{(2)}(r,R) + \mathcal{O}(\varepsilon^4)
\end{equation*}
so that $\Delta_{3n}\mathbf{u}_{3n}$ becomes $\Delta_{3n}\mathbf{u}_{3n} + \varepsilon (\mathcal{D}_{3n} + \mathcal{D}_{-3n})\hat{\mathcal{D}}_{0}\mathbf{u}_{3n} + \varepsilon^{2}\hat{\Delta}_{0}\mathbf{u}_{3n}$ and \eqref{e:RD;Taylor-n} becomes
\begin{subequations}\label{e:RD;asymp}
    \begin{align}
        \mathcal{O}(\varepsilon):& &\qquad \left(\mathbf{M}_1 - \Delta_{3n}\right)\mathbf{u}_{3n}^{(0)} ={}& \mathbf{0},\label{e:RD;asymp-1}\\
        \mathcal{O}(\varepsilon^2):& &\qquad \left(\mathbf{M}_1 - \Delta_{3n}\right)\mathbf{u}_{3n}^{(1)} ={}& \left( \mathcal{D}_{3n} + \mathcal{D}_{-3n}\right)\hat{\mathcal{D}}_{0}\mathbf{u}_{n}^{(0)},\label{e:RD;asymp-2}\\
        \mathcal{O}(\varepsilon^3):& &\qquad \left(\mathbf{M}_1 - \Delta_{3n}\right)\mathbf{u}_{3n}^{(2)} ={}& \left(\hat{\Delta}_{0} - \hat{\mu}\mathbf{M}_{2}\right)\mathbf{u}_{3n}^{(0)} + \left( \mathcal{D}_{3n} + \mathcal{D}_{-3n}\right)\hat{\mathcal{D}}_{0}\mathbf{u}_{3n}^{(1)}\nonumber\\
        & & & - \hat{\nu}\sum_{i+j=n} \mathbf{Q}(\mathbf{u}^{(0)}_{3i},\mathbf{u}^{(0)}_{3j}) - \sum_{i+j+k=n} \mathbf{C}(\mathbf{u}^{(0)}_{3i},\mathbf{u}^{(0)}_{3j},\mathbf{u}^{(0)}_{3k}).\label{e:RD;asymp-3}
    \end{align}
\end{subequations}
We solve \eqref{e:RD;asymp-1} with
    \begin{equation}\label{soln:RD;u0}
        \mathbf{u}_{3n}^{(0)} = A(R) J_{3n}(k_c\,r)\hat{U}_0 + \overline{A(R)} J_{-3n}(k_c\,r)\hat{U}_0,
    \end{equation}
    and \eqref{e:RD;asymp-2} with 
\begin{equation}\label{soln:RD;u1}
\begin{split}
    \mathbf{u}_{3n}^{(1)} ={}& \left[A_1 J_{3n}(k_c\,r) + \overline{A_1} J_{-3n}(k_c\,r)\right]\hat{U}_0\\
    &\quad + \frac{1}{k_c^2}\left[\partial_{R}A \left(\mathcal{D}_{3n} + \mathcal{D}_{-3n}\right) J_{3n}(k_c\,r) + \partial_{R}\overline{A} \left(\mathcal{D}_{3n} + \mathcal{D}_{-3n}\right) J_{-3n}(k_c\,r)\right]\hat{U}_1,
\end{split}
    \end{equation}
    where we note that
    \begin{equation*}
\begin{split}
    \left(\mathbf{M}_1 - \Delta_{3n}\right)[\left(\mathcal{D}_{3n} + \mathcal{D}_{-3n}\right) J_{3n}(k_c\,r)]\hat{U}_{1} ={}& k_c^2 [\left(\mathcal{D}_{3n} + \mathcal{D}_{-3n}\right) J_{3n}(k_c\,r)]\hat{U}_{0}\\
    & -(\Delta_{3n} + k_c^2)[\left(\mathcal{D}_{3n} + \mathcal{D}_{-3n}\right) J_{3n}(k_c\,r)]\hat{U}_{1},\\
    ={}& k_c^2 \left(\mathcal{D}_{3n} + \mathcal{D}_{-3n}\right) J_{3n}(k_c\,r)\hat{U}_0 + \mathcal{O}(\varepsilon^2).\\
\end{split}
\end{equation*}
 Here we have used the fact that $\Delta_{n} = \Delta_{m} + \frac{(m^2-n^2)}{r^2}$ and $(\Delta_{3n\mp1} + k_c^2)\mathcal{D}_{\pm3n}J_{3n}(k_c r) = 0$, and so
\begin{equation*}
    \begin{split}
    (\Delta_{3n} + k_c^2)\left(\mathcal{D}_{3n} + \mathcal{D}_{-3n}\right) J_{3n}(k_c\,r) ={}& k_c (\Delta_{3n} + k_c^2) J_{3n-1}(k_c\,r) - k_c (\Delta_{3n} + k_c^2) J_{3n+1}(k_c\,r) \\
    ={}& \frac{1}{r^2}\left[\left(\mathcal{D}_{3n} + \mathcal{D}_{-3n}\right)J_{3n}(k_c r) - \frac{(6n)^2}{r}J_{3n}(k_c r)\right]    
    \end{split}
\end{equation*}
Evaluating on the slow spatial scale $R$ we conclude that $(\Delta_{3n} + k_c^2)\left(\mathcal{D}_{3n} + \mathcal{D}_{-3n}\right) J_{3n}(k_c\,r) = \mathcal{O}(\varepsilon^2)$, since $\frac{1}{r^2} = \varepsilon^2\,\frac{1}{R^2}$, and hence this term is absorbed into the $\mathcal{O}(\varepsilon^4)$ equation, which does not affect our analysis. 

We substitute \eqref{soln:RD;u0} and \eqref{soln:RD;u1} into \eqref{e:RD;asymp-3}, resulting in 
\begin{equation*}
    \begin{split}
\left(\mathbf{M}_1 - \Delta_{3n}\right)\mathbf{u}_{3n}^{(2)} ={}& \left(\hat{\Delta}_{0} - \hat{\mu}\mathbf{M}_{2}\right)\mathbf{u}_{3n}^{(0)} + \left(\mathcal{D}_{3n} + \mathcal{D}_{-3n}\right)\hat{\mathcal{D}}_{0}\mathbf{u}_{3n}^{(1)} \\
& - \hat{\nu}\sum_{i+j=n} \mathbf{Q}(\mathbf{u}^{(0)}_{3i},\mathbf{u}^{(0)}_{3j}) - \sum_{i+j+k=n} \mathbf{C}(\mathbf{u}^{(0)}_{3i},\mathbf{u}^{(0)}_{3j},\mathbf{u}^{(0)}_{3k}),\\
={}& \left(\hat{\Delta}_{0} - \hat{\mu}\mathbf{M}_{2}\right)\hat{U}_0\left( A J_{3n}(k_c\,r) + \overline{A} J_{-3n}(k_c\,r)\right)\\
& + \left[ \partial_{R}A_1 \left(\mathcal{D}_{3n} + \mathcal{D}_{-3n}\right) J_{3n}(k_c\,r)  + \partial_{R}\overline{A_1} \left(\mathcal{D}_{3n} + \mathcal{D}_{-3n}\right) J_{-3n}(k_c\,r)\right]\hat{U}_0 \\
& + \left[\partial^2_{R}A \left(-4J_{3n}(k_c\,r) + \mathcal{O}(\varepsilon)\right) + \partial_{R}^2\overline{A} \left( - 4J_{-3n}(k_c r) + \mathcal{O}(\varepsilon)\right)\right]\hat{U}_1 \\
& - \hat{\nu}\,\mathbf{Q}_{0} A^2 \sum_{i+j=n} J_{3i}(k_c\,r)J_{3j}(k_c\,r) - 2 \hat{\nu}\,\mathbf{Q}_{0} |A|^2 \sum_{i+j=n} J_{3i}(k_c\,r) J_{-3j}(k_c\,r) \\
& - \hat{\nu}\,\mathbf{Q}_{0} \overline{A}^2 \sum_{i+j=-n} J_{3i}(k_c\,r) J_{3j}(k_c\,r) - \mathbf{C}_{0} A^3 \sum_{i+j+k=n}J_{3i}(k_c r)J_{3j}(k_c r) J_{3k}(k_c r)\\
&  - 3 \mathbf{C}_{0}|A|^2 A \sum_{i+j+k=n} J_{3i}(k_c r) J_{3j}(k_c r) J_{-3k}(k_c r)\\
& - 3 \mathbf{C}_{0}|A|^2 \overline{A}\sum_{i+j+k=-n} J_{3i}(k_c r) J_{3j}(k_c r) J_{-3k}(k_c r)\\
&- \mathbf{C}_{0}\overline{A}^3 \sum_{i+j+k=-n} J_{3i}(k_c r) J_{3j}(k_c r)J_{3k}(k_c r),
    \end{split}
\end{equation*} 
where we have defined $\mathbf{Q}_{0}:=\mathbf{Q}(\hat{U}_0,\hat{U}_0)$ and $\mathbf{C}_{0}:=\mathbf{C}(\hat{U}_0,\hat{U}_0,\hat{U}_0)$. Again the $\mathcal{O}(\varepsilon)$ terms are absorbed into the $\mathcal{O}(\varepsilon^4)$ equation. We take the inner product of $(\mathbf{M}_1 - \Delta_{3n})\mathbf{u}_{3n}^{(2)}$ with $\hat{U}_1^*$ and obtain the following solvability condition
\begin{equation}\label{e:Ginz;RD-Hex}
0 ={} -4 \partial^2_{R}A_h -\hat{\mu}\left(\hat{U}^*_1\cdot\mathbf{M}_{2}\hat{U}_0\right) A_h - 2\hat{\nu}\,\left(\hat{U}^*_1\cdot\mathbf{Q}_{0}\right)\overline{A_h}^2 - 15\left(\hat{U}^*_1\cdot\mathbf{C}_{0}\right) |A_h|^2 A_h,
\end{equation}
where we have again defined $A_h:=\frac{1}{3}A$ and used \eqref{ident:hex;Besselsums} to simplify the convolutional sums. Thus, we obtain an amplitude equation that is qualitatively the same as \eqref{e:Ginz-hex} for the Swift--Hohenberg equation. 

\section{Conclusion}\label{s:Discussion}
In this paper, we have derived a new multiple-scales asymptotic analysis to explain the emergence of fully localised planar patterns. We first posed the Swift--Hohenberg equation in polar coordinates before carrying out a Fourier expansion in the angular variable. By introducing a slow variable $R=\varepsilon r$ and performing a formal multiple-scales analysis, we found a slowly-varying axisymmetric amplitude over a domain-covering pattern described by an infinite sum of Bessel functions. Here, each Bessel function replaces the standard exponential function in our multiple-scales analysis, and so we developed appropriate tools and techniques to work with derivatives and nonlinear combinations of these functions. Then, we were able to carry out the necessary asymptotic analysis in order to derive a single amplitude equation, for which we could find explicit localised solutions. This approach helps to explain numerical observations that, at small amplitude, fully localised patches appear to consist of an axisymmetric envelope over a domain-covering pattern. 

This approach has several advantages over other attempts to explain the existence of localised patches of patterns. The main advantage is the simplicity of the approach and its generality. In each of the cases presented, we obtain a single equation for the slow-varying amplitude, to which we are able to find explicit localised solutions. The method can easily be extended to general reaction-diffusion systems, as we have shown in \S\ref{s:RD}, and we believe it can be easily extended to more complicated models like Rayleigh--B\'enard convection. 

Our results raise new questions regarding the connections between fully localised patterns and the localised axisymmetric patterns studied in \cite{Lloyd2009LocalizedRadial,McCalla2013Spots}. In the Galerkin finite-Fourier approach of Hill et al.~\cite{Hill2023DihedralSpots,Hill2024DihedralRings}, it was found that some localised dihedral patterns (including hexagons) are qualitatively similar to axisymmetric spots, while others were instead related to axisymmetric rings. We note, however, that the approach described in this paper does not work for purely axisymmetric patterns, since we do not have any equivalent identities for dealing with nonlinear products of just $J_0$ Bessel functions. Furthermore, it is not clear that one would be able to recover our results in a Galerkin finite-Fourier approximation, since we again lose the nonlinear identities that are vital to our analysis. Notably, Remark \ref{rmk:rotate} tells us that arbitrarily rotated hexagons possess the amplitude equation \eqref{e:Ginz-D12} for all rotations $\alpha\in(0.\frac{\pi}{6})$, suggesting that the quadratic-cubic Ginzburg--Landau equation seen in \eqref{e:Ginz-hex} and \eqref{e:Ginz-D12} is highly generic.

Several extensions and problems remain. For instance, it would be interesting to explore radial invasion fronts, where a circular region of pattern widens and invades the trivial state; see~\cite{Castillo-Pinto2019RadialInvasion} for an investigation in this area. Since each amplitude equation is posed on a half-line with continuity condition $\partial_{R}A(0)=0$, it is unclear whether one can go to a travelling frame like in the study of invading one-dimensional fronts. However, translating patches of pattern should still be tractable with the approach described in this paper, since one can go to a travelling frame before then introducing polar coordinates. We note that an amplitude equation does not explain the snaking or non-snaking of the patches observed in numerics and so extending the exponential asymptotics of Kozyreff \& Chapman~\cite{KozyreffChapman2013Hexagon} would be an interesting next step. Finally, the rigorous validation of the amplitude equation for long time scales following the approach in~\cite{Schneider1994ErrorEstimatesGL} would also be worth investigating.

\section*{Acknowledgements}
DJH gratefully acknowledges support from the Alexander von Humboldt Foundation. The idea for this work began during a Postdoc/PhD seminar by Bastian Hilder and Jonas Jansen at Institut Mittag--Leffler during the research program `Order and Randomness in Partial Differential Equations', supported by the Swedish Research Council under grant no. 2021-06594 in the winter semester of 2023.

\bibliographystyle{abbrv}
\bibliography{Bibliography.bib}

\end{document}